\documentclass[reqno,11pt]{amsart}
\usepackage{amsmath,amssymb,amsfonts,amsthm, amscd,indentfirst}
\usepackage{amsmath,latexsym,amssymb,amsmath,%awheremscd,
	amscd,amsthm,amsxtra}
\usepackage{color}
\usepackage{pdfpages}
\usepackage{graphics}
\usepackage{subfig}
\usepackage{epstopdf}
\usepackage{epsfig}
\usepackage{enumitem}
\usepackage{float}
\usepackage[usenames,dvipsnames]{xcolor}

\usepackage[msc-links]{amsrefs} %--reference with MR, MATHSCI format--

\usepackage[T1]{fontenc}
\usepackage{newtxtext,newtxmath}

\usepackage{varioref}
\usepackage{hyperref}
\usepackage[noabbrev,capitalise,nameinlink]{cleveref}

\hypersetup{colorlinks={true},linkcolor={blue},citecolor=blue,urlcolor=blue}
\usepackage[english]{babel}% Recommended
\usepackage{csquotes}% Recommended

\usepackage{graphicx}
\usepackage{caption}

\newtheorem{theorem}{Theorem}[section]%[section]
\newtheorem{lemma}[theorem]{Lemma}%[section]
\newtheorem{proposition}[theorem]{Proposition}%[section]
\newtheorem{corollary}[theorem]{Corollary}%[section]
\newtheorem{remark}[theorem]{Remark}

\def\O{\Omega}

\def\S{\Sigma}
\def\n{\nabla}

\def\p{\partial}

\def\a{\alpha}

\def\n{\nabla}

\def\mfk{\mathbf{k}}

\def\O{\Omega}
\def\p{\partial}

\def\a{\alpha}

\def\n{\nabla}
\def\<{\langle}
\def\>{\rangle}

\def\n{\nabla}

\def\O{\Omega}
\def\p{\partial}

\def\a{\alpha}

\def\R{\mathbb{R}}

\def\rr{\mathbb{R}}

\def\mbR{\mathbb{R}}
\def\mcH{\mathcal{H}}
\def\mcP{\mathcal{P}}

\def\mfW{\mathbf{W}}

\newcommand{\rd}{{\rm d}}

\begin{document}
	
	%	\title[A Heintze-Karcher-type inequality]{A Heintze-Karcher-type inequality for hypersurfaces with capillary boundary II}
	
	\title[Heintze-Karcher inequality and capillary hypersurfaces]{Heintze-Karcher inequality and \\capillary hypersurfaces in a wedge}
	
	\author[Jia]{Xiaohan Jia}
	\address[X.J]{School of Mathematics\\
Southeast University\\
 211189, Nanjing, P.R. China}
	\email{jia92@mail.ustc.edu.cn}
	
	\author[Wang]{Guofang Wang}
	\address[G.W]{Mathematisches Institut\\
		Universit\"at Freiburg\\
	Ernst-Zermelo-Str.1,
		79104, Freiburg, Germany}
	\email{guofang.wang@math.uni-freiburg.de}
	
	\author[Xia]{Chao Xia}
	\address[C.X]{School of Mathematical Sciences\\
		Xiamen University\\
		361005, Xiamen, P.R. China}
	\email{chaoxia@xmu.edu.cn}
	
	\author[Zhang]{Xuwen Zhang}
	\address[X.Z]{School of Mathematical Sciences\\
		Xiamen University\\
		361005, Xiamen, P.R. China}
	\email{xuwenzhang@stu.xmu.edu.cn}
	\thanks{This work is  supported by the NSFC (Grant No. 11871406, 12271449, 12126102)}
	%	\thanks{This work is  supported by test.
		%	}

	\begin{abstract}
	%This paper is a continuation of the previous work \cite{JXZ22}.
{In this paper, we utilize the  method of Heintze-Karcher to prove a "best" version of Heintze-Karcher-type inequality for capillary hypersurfaces in the half-space or in  a wedge.  
 One of new crucial ingredients in the proof is  modified parallel hypersurfaces which are very natural to be used to study capillary hypersurfaces. A more technical part is a subtle analysis along  the edge of a wedge.
As an application, we classify completely  embedded capillary constant mean curvature hypersurfaces that hit the edge in a wedge, which is a subtler case.
}    %, and to get rid of the restriction of capillary angle. Another purpose is to derive the Heintze-Karcher-type inequality when the ambient domain is a wedge.
		
		\
		
		
	\end{abstract}
	
	\maketitle
	
	\medskip

\section{Introduction}

%\textcolor{red}{A wedge is determined by planes is better than its boundary consisting of planes}

The study of capillary surfaces goes back to  Thomas Young, who 
studied in 1805 the
equilibrium state of liquid fluids. It was he who  first introduced  the notion of mean curvature and 
the boundary contact angle condition of capillarity,  the so-called Young’s law.
This problem was  reintroduced and reformulated by  Laplace and by  Gauss  later.
For the history of  capillary surfaces  see Finn's survey \cite{Finn99}. 
A capillary  hypersurface $\S$ in $\R^{n+1}$ with boundary $\p\Sigma$  on a support hypersurface $S\subset \R^{n+1}$ is a critical point of the following
functional
$$|\Sigma |-\cos \theta |\bar \O\cap S| $$
among all compact hypersurfaces with boundary $\p\Sigma$  on $S$ under a  volume constraint. Here $\O$ is the bounded domain enclosed by $\S$ and $S$, and $|\Sigma|$
is the $n$-dimensional area of $\Sigma$,
and $\theta \in (0, \pi)$. Equivalently, a capillary hypersurface is a constant mean curvature (CMC) hypersurface with boundary which intersects the support hypersurface $S$ at a constant angle $\theta$. %it satisfies 
%\begin{eqnarray} \label{1.1}
%H&=& c, \quad \hbox{ in } \Sigma\\ \label{1.2}
%\langle \nu, \bar N\rangle &=& \cos \theta, \quad \hbox{ on } \p\Sigma\
%\end{eqnarray} where $H$ is the mean curvature, $\nu$  the outward unit  normal of $\Sigma \subset \R^{n+1}$ and 
% $\bar N$ is one of the   unit normal vector fields  of $S\subset \R^{n+1}$. 
There has been a lot of 
 interdisciplinary
investigations on the stationary solutions and local minimizers of the above energy. 
For the interested reader, we refer to  Finn’s book  \cite{Finn86},
which is  an excellent  survey on capillary surfaces.

Inspired  by the recent development of the min-max theory for minimal surfaces and CMC surfaces
\cite{CD03, MN16a, MN16b}, there have been a lot of works on free boundary minimal surfaces and CMC surfaces, which are a special class of capillary surfaces with $\theta=\frac{\pi}{2}$, see for example \cite{DR18, LZ21, GLWZ21, ZZ19}. Very recently the min-max theory for capillary surfaces was developed in \cite{DMDP21, LiZZ21}. 

One important application of the capillary surfaces was recently obtained by Chao Li in \cite{Li21}, where he utilizes capillary surfaces in a polyhedron to study Gromov's dihedral rigidity conjecture. His work on capillary surfaces is related to this paper, in which we will consider capillary surfaces supported on a wedge. %See also the remark at the end of introduction.

The main objective of this paper is {to make a complete classification of embedded capillary hypersurfaces in a wedge. Such a hypersurface can be seen as a model for capillary hypersurfaces in Riemannian polyhedra.
} For the precise definition of a wedge, see below in the Introduction. Our starting point is 
 the Heintze-Karcher inequality. Let us first recall Heintze-Karcher's theorem and Heintze-Karcher's inequality
for closed hypersurfaces.

In a seminal paper \cite{HK78},  Heintze-Karcher proved a general tubular volume comparison theorem for embedded Riemannian submanifolds, which generalizes the celebrated Bishop-Gromov's volume comparison theorem in Riemannian geometry.
For an embedded closed hypersurface $\S$, which encloses a bounded domain $\O$, in an $(n+1)$-dimensional Riemannian manifold of nonnegative Ricci curvature,  Heintze-Karcher's theorem reads as follows,
\begin{eqnarray}\label{HK1}
|\O|\le \int_\S \int_0^{c(p)}\left(1-\frac{H(p)}{n}t\right)^n \rd t\rd A.
\end{eqnarray}
Here $H(p)$ is the mean curvature of $\S$ at $p$ and $c(p)$ is the length to reach the first focal point of $\S$ from $p$ by the normal exponential map. 
As a direct consequence of \eqref{HK1}, one deduces that 
\begin{eqnarray}\label{HK2}
|\O|\le \frac{n}{n+1}\int_\S \frac{1}{H}\rd A,
\end{eqnarray}
provided that $\S$ is strictly mean convex, namely, $H>0$ on $\S$. 
Nowadays, \eqref{HK2} is literately referred to as Heintze-Karcher's inequality in  hypersurfaces theory. A well-known new proof  via Reilly's formula \cite{Reilly77} has been given by Ros \cite{Ros87}. Moreover, Ros \cite{Ros87} utilized the Heintze-Karcher inequality \eqref{HK2} to reprove the celebrated Alexandrov's soap bubble theorem, which states that any embedded closed constant mean curvature hypersurfaces in $\mathbb{R}^{n+1}$ must be a round sphere.

Since then various Heintze-Karcher-type inequalities have been established in various circumstance. For instance, Montiel-Ros \cite{MR91} and Brendle \cite{Br13} established  Heintze-Karcher-type inequalities in space forms and in certain warped product manifolds respectively, see also \cite{QX15, LX19}. {The Heintze-Karcher inequality in $\mathbb{R}^{n+1}$ has been also established for sets of finite perimeter, see e.g. \cite{DM19,San19}}. Like the Alexandrov-Fenchel inequalities, the Heintze-Karcher inequality  becomes one of fundamental geometric inequalities in differential geometry.

Inspired by the method of Ros \cite{Ros87}, and also by the work of Brendle \cite{Br13}, we have proved a Heintze-Karcher-type inequality for hypersurfaces with free boundary in a unit ball 
\cite{WX19} by using a generalized Reilly's  formula proved by Qiu-Xia \cite{QX15}. However this method leads to a slight different inequality if we consider hypersurfaces with capillary boundary in the unit ball or in the half-space in the very recent work, \cite{JXZ22}. Precisely
we have established in the previous work \cite{JXZ22} a version of Heintze-Karcher-type inequality for hypersurfaces in the half-space with capillary boundary, by using the solution to a mixed boundary value problem in the classical Reilly's formula. Let $\rr^{n+1}_+:=\{x\in \rr^{n+1}:\<x, E_{n+1}\>>0\}$, where $E_{n+1}=(0,\cdots, 0, 1)$, and for an embedded, compact, strictly mean-convex hypersurface $\Sigma\subset \overline{\rr^{n+1}_+}$ with capillary boundary
with a constant contact angle $\theta_0\in(0,\frac{\pi}{2}]$, there holds
\begin{align}\label{EQ-HK-halfspace0}
			\int_\S\frac{1}{H} \rd A\ge \frac{n+1}{n} |\O|+\cos\theta_0\frac{\left(\int_\S \<\nu, E_{n+1}\>\rd A\right)^2}{\int_\S H\<\nu, E_{n+1}\> \rd A},
		\end{align}
with equality if and only if $\S$ is a spherical cap. 

%Here  $\S$ is called strictly mean convex, if the mean curvature $H>0$ on $\S$.

%As we have commented in \cite{JXZ22}, although 
Inequality \eqref{EQ-HK-halfspace0} is optimal, in the sense that  the spherical caps achieve equality in \eqref{EQ-HK-halfspace0}. However it is not in the best form. 
For example, while we are able to use \eqref{EQ-HK-halfspace0} to reprove the Alexandrov theorem for constant mean curvature (CMC) hypersurfaces in \cite{JXZ22}, it is not very helpful to handle the case of higher order mean curvatures (see \eqref{defn-Pn(t)} for the definition).
%it should not be the 
%optimal version when $\theta\in (0,\frac{\pi}{2})$, 
In view of the following Minkowski formula
\begin{eqnarray}\label{mink-formula_1}
\int_\S n(1-\cos\theta_0\<\nu, E_{n+1}\>)-H\<x,\nu\> \rd A=0,
\end{eqnarray} 
a possible best form, which was conjectured in \cite{JXZ22}, is
\begin{align}\label{EQ-HK-halfspace_conj}
     \int_\S\frac{1-\cos\theta_0\left<\nu,E_{n+1}\right>}{H}\rd A
     \geq\frac{n+1}{n}\vert\O\vert.
    \end{align}
 %   which 
   % \begin{eqnarray}\label{mink-formula} \int_\S n(1-\cos\theta_0\<\nu, E_{n+1}\>)-H\<x,\nu\> dA=0. \end{eqnarray} 
%For example, while we are able to use \eqref{EQ-HK-halfspace0} to %reprove the Alexandrov theorem for constant mean curvature (CMC) %hypersurfaces, it is not very helpful to handle the case of higher order mean curvatures.
It is clear by using the Cauchy-Schwarz inequality that inequality \eqref{EQ-HK-halfspace_conj} implies \eqref{EQ-HK-halfspace0}, provided  that $\<\nu, E_{n+1}\> $ is non-negative. But without the non-negativity one does not know which one is stronger.
%{\color{red} It seems not true in general, one needs to assume $<\nu,E_{n+1}>\ge 0$ so as to use Cauchy-Schwarz.}

The first part of  this paper is to establish this ``best'' version of the Heintze-Karcher inequality in a little more general setting and for whole range $\theta_0\in (0,\pi)$. Let $\Sigma$ be a hypersurface in $\overline{\mbR^{n+1}_+}$ with (possibly  non-connected) boundary $\partial \Sigma \subset \partial \rr^{n+1}_+$. The hypersurface intersects with the supported hyperplane $\partial \rr^{n+1}_+$ transversely. 
\footnote{In the paper we abuse a little bit the terminology of capillarity. A capillary hypersurface mentioned at the beginning of the Introduction is  called a capillary CMC  hypersurface in the paper.}

%Moreover, we can handle all range of $\theta\in(0,\pi)$ and also the case that the capillary angle $\theta$ is a function $\theta: \p\S\to (0,\pi)$ such that $\theta(x)\le \theta_0$ for some $\theta_0$. We call $\S$ is a $\theta(x)$-capillary hypersurface in some container $\overline{\mbR^{n+1}_+}$ for $\theta:\p\S\to (0,\pi)$ if it intersects $\p \rr_+^{n+1}$ at contact angle $\theta(x)$ at $x\in \p\S$. 
\begin{theorem}\label{Thm-HK-halfspace}
   Let $\theta_0\in(0, \pi)$ and let $\S\subset\overline{\mbR^{n+1}_+}$ be a smooth, compact, embedded, strictly mean convex $\theta$-capillary hypersurface, with  $\theta(x)\le \theta_0$ for every $x\in \partial\Sigma$. Let $\O$ denote the enclosed domain by $\S$ and $\p \mbR^{n+1}_+$. Then it holds
    \begin{align}\label{EQ-HK-halfspace}
     \int_\S\frac{1-\cos\theta_0\left<\nu,E_{n+1}\right>}{H}\rd A
     \geq\frac{n+1}{n}\vert\O\vert.
    \end{align}
    Equality in \eqref{EQ-HK-halfspace} holds if and only if $\S$ is a $\theta_0$-capillary spherical cap.
\end{theorem}
Note that we also have removed the restriction that $\theta_0 \le \frac \pi 2$, comparing with the previous work \cite{JXZ22}.
 As a direct application, we get the Alexandrov-type theorem for embedded capillary hypersurfaces in $\overline{\mbR^{n+1}_+}$ with constant $r$-th mean curvature, for any $r\in\{1,\cdots, n\}$.%which indicates that \eqref{EQ-HK-halfspace} is more powerful than \eqref{EQ-HK-halfspace0}
	\begin{corollary}\label{alexandrov-halfspace}
		Let $\theta_0\in(0,\pi)$ and $r\in\{1, \cdots,n\}$. Let $\S\subset \overline{\rr^{n+1}_+}$ be a smooth, embedded, compact,  $\theta_0$-capillary hypersurface with constant $r$-th mean curvature. Then
	$\S$ is a $\theta_0$-capillary spherical cap.\end{corollary}

Our proof of \cref{Thm-HK-halfspace} is inspired by the original idea of Heintze-Karcher \cite{HK78} (see also Montiel-Ros \cite{MR91}) which  uses parallel hypersurfaces to estimate the enclosed volume. 
However the ordinary  parallel hypersurfaces do not work  for capillary hypersurfaces. 
The one of key ingredients of this paper  is a correct form of parallel hypersurfaces $\zeta(\S, t)$ defined in \eqref{parallel}. To prove \cref{Thm-HK-halfspace}, we need to show  the surjectivity of $\zeta$ onto the enclosed domain $\Omega$, for which we discover an appropriate foliation by round spheres with simultaneously varied center and radius.

It is interesting to see that our proof of \cref{Thm-HK-halfspace}
provides a refinement of the ordinary Heintze-Karcher inequality for closed hypersurfaces, since any closed hypersurface can be viewed as a capillary hypersurface with an empty boundary. Hence we have
\begin{corollary}
Let $\Sigma$ be a closed, strictly mean convex hypersurface in ${\mathbb R}^{n+1}$ with enclosed domain $\Omega$. Then it holds
\begin{eqnarray*}
    \int_\Sigma \frac 1 H\rd A - \max_{e\in {\mathbb S}^n}\left| 
    \int_\Sigma \frac {\langle v, e\rangle} H  \rd A 
    \right| \ge 
    \frac  {n+1}  n |\Omega|.
\end{eqnarray*}
Equality in \eqref{Thm-HK-halfspace} holds if and only if $\S$ is a round sphere.
\end{corollary}

\

In the second part of this paper, we study  hypersurfaces with capillary boundary in a wedge domain. Here we simply call it  a wedge.
{An ordinary wedge, we call it a classical wedge in this paper,  is the unbounded closed region determined by two intersecting hyperplanes with dihedral angle $\alpha$, which is also called an opening angle, lying in $(0,\pi)$.} 
There have been many works on the study of the stability of CMC capillary hypersurfaces (c.f. \cite{CK16, LX17, Souam21,XZ21}) and on embedded CMC capillary hypersurfaces in wedges (c.f. \cite{McCuan97, Park05, Lopez14}).  Comparing with the half-space case, a big difference is that the Alexandrov's reflection method might fail in the case of wedges, though the authors   in \cite{McCuan97, Park05, Lopez14}   managed to modify Alexandrov's reflection to obtain their classification results in certain cases.
It is interesting that our method to establish the Heintze-Karcher-type inequality works in the wedge case, and even works in a more general setting.
{\color{black}
See also the recent development of this method in the anisotropic setting \cite{JWXZ23,JWXZ23b}.

}

In fact  we shall consider generalized wedges which are determined by  finite many mutually intersecting hyperplanes.
 To be more precise, let $\mfW$ be the unbounded closed region in $\mathbb{R}^{n+1} (n\ge 2)$,
 which are determined by finite many mutually intersecting hyperplanes $P_1,\ldots, P_L$, for some integer $1\le L\le n+1$, such that the dihedral angle between $P_i$ and $P_j$, $i\ne j$, lies in $(0,\pi)$.
% (\textcolor{red}{Here $P_i$ is certainly only a part of hyperplane. We will not distinguish this point. Since now we use "determined by hyperplanes", I will remove this remark})}
 %whose boundary $\partial \mfW$ consists of finite many mutually intersecting hyperplanes $P_1,\ldots, P_L$, for some integer $1\le L\le n+1$, such that the dihedral angle between $P_i$ and $P_j$, $i\ne j$, lies in $(0,\pi)$. 
 We call such $\mfW$ a generalized wedge. $P_i\cap P_j, i\ne j$ is called an edge of the wedge $\mfW$. If $L=2$, we call $\mfW$ a classical wedge. Let $\bar N_i$ be the outwards pointing unit normal to $P_i$ in $\mfW$ for $i=1,\cdots, L$. Thus $\{\bar N_{1},\ldots,\bar N_{L}\}$ are linearly independent. Up to a translation, we may assume that the origin $O\in \bigcap\limits_{i=1}^LP_i$.
 Given $\vec{\theta}_0=(\theta^1_0,\cdots,\theta^L_0)\in \prod\limits_{i=1}^L(0,\pi)$. Now we define an important vector $\mfk_0$ associated with $\mfW$  and  $\vec{\theta}_0$ by 
\begin{align}\label{k0}
\mfk_0=\sum_{i=1}^L c_i\bar N_i,
\end{align}
where $c_i$ is such that $\<\mfk_0, \bar N_i\>=\cos\theta_0^i$.
% It is clear that $\mfk_0$ is well-defined.
We say that $\S$ is a $\vec{\theta}$-capillary hypersurface in $\mfW$ with $\vec{\theta}=(\theta^1, \theta^2, \cdots , \theta^L) $ if it intersects $\p \mfW$ at contact angle $\theta^i(x)$ for $x\in \p\S\cap P_i$. 
It is easy to see that $\mfk_0$ is the center of the $\vec\theta_0$-capillary spherical cap with radius $1$.
The following key assumption \eqref{wedge-condition} has a clear geometric meaning that the unit sphere  centered at $\mfk_0$ intersects the edge of the wedge $\mfW$.

We shall prove the following Heintze-Karcher-type inequality in a wedge.
\begin{theorem}\label{Thm-HK-wedge}
    Let $\mfW\subset\mbR^{n+1}$ be a generalized wedge whose boundary consists of $L$ mutually intersecting closed hyperplanes $\{P_i\}_{i=1}^L$ and $\vec{\theta}_0\in\prod\limits_{i=1}^L(0,\pi)$. Assume that 
    \begin{align}\label{wedge-condition}
        \vert\mfk_0 \vert\leq 1.
    \end{align} Let $\S\subset \mfW$ be a smooth, compact, embedded, strictly mean convex $\vec{\theta}$-capillary hypersurface %, where $\vec{\theta}:\p\S\to \prod\limits_{i=1}^L(0,\pi)$ is such that 
    with $\theta^i(x)\le \theta_0^i$ for $x\in \partial \Sigma\cap P_i$, $i=1,\ldots, L$. Let $\O$ be the enclosed domain by $\S$ and $\p\mfW$. Assume in addition that
    $\S$ does not hit the edges of $\mfW$, i.e., \begin{align}\label{away-edge}\S\cap (P_i\cap P_j)=\emptyset, \quad i\neq j.\end{align}
 Then 
    \begin{align}\label{EQ-HK-wedge}
        \int_\S\frac{1+\left<\nu,\mfk_0\right>}{H}\rd A
        \geq\frac{n+1}{n}\vert\O\vert,
    \end{align}
    with equality if and only if %in \eqref{EQ-HK-wedge} holds, then 
    $\S$ is a $\vec{\theta}_0$-capillary spherical cap.
\end{theorem}

The idea of proof of \cref{Thm-HK-wedge} is similar to that of \cref{Thm-HK-halfspace}, by using a suitable family of parallel hypersurfaces $\zeta(\S, t)$ which relates $\mfk_0$.  
To show the surjectivity of $\zeta$, assumption \eqref{away-edge} plays a crucial role.
%Due to significance of the physical phenomenon, we focus on the case that $\mfW$ is a classical wedge ($L=2$) in $\mathbb{R}^3$. In this case,
%-----------------
    %On the other hand, 
Actually, this condition was required in %\textcolor{red}{all?}
previous related papers, except \cite{Lopez14}. See \cref{remark1.10}.
In this paper we are able to remove the additional assumption \eqref{away-edge}
%by virtue of a subtle analysis on the edge 
in a classical wedge, i.e., $L=2$.
%In the case that $\mfW\subset \mathbb{R}^{n+1}$ is a classical wedge, i.e., $L=2$,
%by virtue of a subtle analysis on the edge $P_1\cap P_2$, we are able to remove the additional assumption \eqref{away-edge}.
Precisely, we have the following
%-----------------
%\begin{theorem}\label{Thm-HK-wedge-2} other    Let $\mfW\subset\mbR^{n+1}$ be a  wedge whose boundary consists of $P_i, i=1, 2$, and  $\vec{\theta}_0\in\prod\limits_{i=1}^2(0,\pi)$ such that \eqref{wedge-condition} holds. Let $\S\subset \mfW$ be a smooth, compact, embedded, strictly mean convex $\vec{\theta}(x)$-capillary hypersurface, where $\vec{\theta}:\p\S\to \prod\limits_{i=1}^2(0,\pi)$ is such that $\theta^i\le \theta_0^i$ for $i=1, 2$. Let $\O$ be the enclosed domain by $\S$ and $\p\mfW$. \noindent Then \eqref{EQ-HK-wedge} holds with equality holding if and only if $\S$ is a $\vec{\theta}_0$-capillary spherical cap. \end{theorem}
\begin{theorem}\label{Thm-HK-wedge-2}
When $L=2$, \cref{Thm-HK-wedge} holds true without assumption \eqref{away-edge}.
\end{theorem}
The  proof of \cref{Thm-HK-wedge-2} relies on a delicate analysis on the edge, which is the most technical part of this paper. A special case, $L=2$ and $\vec\theta_0= (\frac \pi 2, \frac \pi 2)$, i.e., $\Sigma$ is a free boundary hypersurface, for which $\mfk_0=0,$ 
\eqref{EQ-HK-wedge} was proved by Lopez in    \cite{Lopez14} via Reilly's formula.
It is a natural question to ask if \cref{Thm-HK-wedge} holds true for $L>2$ without \eqref{away-edge}.
 \cref{Thm-HK-wedge-2} leads us to believe that assumption 
\eqref{away-edge} is unnecessary. 

Now we make some remarks on condition \eqref{wedge-condition}.

\begin{remark}\label{Rem-k0}
 \normalfont\
 
 \begin{itemize} 
 
\item[(i)] In view of the Heintze-Karcher inequality \eqref{EQ-HK-wedge} we have established, \eqref{wedge-condition} could not be removed.

\item[(ii)] When $L=1$, it is nothing but the half-space case and \eqref{wedge-condition} is satisfied automatically. 
When $L=2$, \eqref{wedge-condition} is equivalent to 
\begin{align}\label{wedge-equiv}
    \vert\pi-(\theta_0^1+\theta_0^2)\vert\le\alpha\le\pi-\vert\theta_0^1-\theta_0^2\vert.
   %|\theta_1-\theta_2|\leq\pi-\alpha\leq \theta_1+\theta_2\leq \pi+\alpha
\end{align}
where $\alpha\in (0,\pi)$ is the opening angle of the wedge, or the dihedral angle between $P_1$ and $P_2$, see \cref{lem-condition}. Similarly, $\vert\mfk_0 \vert< 1$ is equivalent to \eqref{wedge-equiv} with strict inequalities.

\item[(iii)]  By virtue of \cref{Lem-k<=1}, Condition \eqref{wedge-condition} is satisfied, provided there exists a $\vec{\theta}_0$-capillary hypersurface $\S$ in $\mfW$ which satisfies $\S\cap P_1\cap P_2\neq\emptyset$. 

\item [(iv)]{Condition \eqref{wedge-condition} is also closely related to the existence of elliptic points of capillary hypersurfaces. See Section 5 below.}

\item[(v)] Condition \eqref{wedge-equiv} appears also in the regularity of capillary surfaces at corner. See the last paragraph of the Introduction.
\end{itemize}
\end{remark}

%{\color{purple} Due to the significance of the physical phenomenon, we are mainly concerned with the classical wedges in $\mathbb{R}^3$, and we will prove: a classical wedge in $\mathbb{R}^3$, under certain geometric condition\eqref{k0}, admits no spherical spanner; on the other hand, we will establish an Alexandrov type theorem for hypersurfaces touching the edge of such wedges.} 
As applications of  \cref{Thm-HK-wedge} and \cref{Thm-HK-wedge-2}, we prove  an Alexandrov-type theorem and a non-existence result for embedded CMC capillary hypersurfaces in a wedge.
\begin{theorem}\label{Thm-Alexandrov}
  Let $\mfW\subset\mbR^{n+1}$ be a classical wedge whose boundary consists of $P_1, P_2$ with $\vec{\theta}_0\in\prod\limits_{i=1}^2(0,\pi)$. Let $\S\subset \mfW$ be a smooth, compact and  embedded $\vec{\theta}_0$-capillary hypersurface with constant $r$-mean curvature, $r\in\{1,\cdots,n\}$. Assume $\S\cap P_1\cap P_2\not=\emptyset$.
 Then $\S$ is a $\vec{\theta}_0$-capillary spherical cap.
\end{theorem}
%\textcolor{red}{Need condition (8) in  Theorem 1.7? }{\color{blue}No. From Remark 1.6 (iii), we see the assumption of Theorem 1.7 already implies (8). }

%\textcolor{red}{Theorem 1.7 holds true for generalized wedges?}{\color{blue}No. For multiple planes, one needs to handle the case $\S\cap P_1\cap P_2\cap P_3\neq\emptyset$, which would be too complicated. We think $L=2$ is enough.}

\begin{theorem}\label{Thm-non-exist}
   Let $\mfW\subset\mbR^{n+1}$ be a classical wedge whose boundary consists of $P_1, P_ 2$ with $\vec{\theta}_0\in\prod\limits_{i=1}^2(0,\pi)$. Then there exists no smooth, compact and  embedded $\vec{\theta}_0$-capillary, CMC hypersurface such that $\S\cap P_1\cap P_2=\emptyset$ and $\vert\mfk_0 \vert\le 1$. Moreover, there exists no smooth, compact, embedded, $\vec{\theta}_0$-capillary  hypersurface of constant $r$-mean curvature for some $r\in \{2,\cdots, n\}$, such that $\S\cap P_1\cap P_2=\emptyset$ and $\vert\mfk_0 \vert< 1$.
\end{theorem}

%\cref{Thm-non-exist} indicates the necessity of condition \eqref{wedge-condition} in the study of capillary hypersurfaces.
 As a consequence of  \cref{Thm-Alexandrov}, \cref{Thm-non-exist} and \cref{Rem-k0} (iii), we have the following
 
 \begin{theorem}\label{Thm-Alexandrov2}
 Let $\mfW\subset\mbR^{n+1}$ be a classical wedge whose boundary consists of $P_1, P_2$ with $\vec{\theta}_0\in\prod\limits_{i=1}^2(0,\pi)$. Let $\S\subset \mfW$ be a smooth, compact and embedded $\vec{\theta}_0$-capillary CMC hypersurface. Then $\S$ is a $\vec{\theta}_0$-capillary spherical cap which intersects with the edge $P_1\cap P_2$ if and only if  $\vert\mfk_0 \vert\le 1$.
\end{theorem}

%\begin{theorem}\label{Thm-Alexandrov2}
% Let $\mfW\subset\mbR^{n+1}$ be a classical wedge whose boundary consists of $P_i, i=1, 2$, and $\vec{\theta}_0\in\prod\limits_{i=1}^2(0,\pi)$ such that $\vert\mfk_0 \vert\le 1$. Let $\S\subset \mfW$ be a smooth, compact and embedded $\vec{\theta}_0$-capillary CMC hypersurface. Then $\S$ is a $\vec{\theta}_0$-capillary spherical cap which intersects with the edge $P_1\cap P_2$.
%\end{theorem}

Several remarks and questions are in order.  
\begin{remark}\normalfont\
\label{remark1.10}
\begin{itemize}

\item[(i)] %Our non-existence result follows, since
 McCuan \cite[Theorem 2]{McCuan97} proved that
any $\vec{\theta}_0$-capillary spherical cap which is disjoint with the edge must satisfy  $\theta_0^1+\theta_0^2> \pi+\alpha$. This condition implies that %see McCuan \cite[Theorem 2]{McCuan97}. However, this violates the angle assumption 
$\vert\mfk_0\vert>1$, see  \cref{lem-condition}. 

    \item[(ii)] Lopez \cite{Lopez14} proved an Alexandrov-type theorem for embedded CMC capillary surfaces with $\vec{\theta}_0=(\frac{\pi}{2},\frac{\pi}{2})$, i.e., the free boundary case. Note that in this case, $\vert\mfk_0 \vert\leq 1$ is automatically satisfied. Hence \cref{Thm-Alexandrov2} covers Lopez's result.

In contrast to it  a ring-type CMC free hypersurface in a wedge was constructed by Wente in \cite{Wente95}, which is certainly not embedded. It is natural to ask whether there exist immersed
 ring-type CMC $\vec{\theta}_0$-hypersurfaces for a general $\vec{\theta}_0$.

 \item[(iii)] Park \cite{Park05} classified the embedded CMC capillary ring-type spanners, which are topologically annuli and disjoint with the edge. Our \cref{Thm-Alexandrov} classified all embedded CMC capillary surfaces intersecting with the edge, without any topological condition.

 \item[(iv)] McCuan \cite{McCuan97} proved a non-existence result for the  embedded CMC capillary ring-type spanners with \begin{align}\label{Mccuan-condition}
    \theta_0^1+\theta_0^2\le \pi+\alpha,
\end{align} by developing spherical reflection technique, when $n=2$. Note that   the angle relation \eqref{Mccuan-condition} is weaker than $\vert\mfk_0 \vert\leq 1$, see  \cref{lem-condition}.
However, our \cref{Thm-non-exist} requires no topological assumption. Moreover it holds for any dimensions. % but with a stronger assumption $\vert\mfk_0 \vert\leq 1$ on the angle relation.

%\item[(v)] It remains an open problem:  If \eqref{wedge-condition} does not hold, i.e.
%        $$\vert\mfk_0 \vert> 1$
\item[(v)]  For stable CMC capillary hypersurfaces in a classical wedge, Choe-Koiso \cite{CK16} proved that such a surface is a part of  a sphere without the angle condition \eqref{wedge-equiv}, but with condition \eqref{away-edge} and with the embeddness of $\p \S$ for $n=2$ or the convexity of $\p\S$ for $n\ge 3$. It is an interesting  question to ask if an immersed stable CMC capillary hypersurface in a wedge is a part of a sphere, without any further conditions, c.f., \cite{WX19, GWX21, Souam21}.

\end{itemize}
\end{remark}

%The study of capillary surfaces is  important not only in physics, but also in mathematics.
We end the Introduction with a few supplement on the study of capillary hypersurfaces in a wedge domain. A nonparametric capillary surface is  a graph of a function $f$ over a domain, say $\O$,  which satisfies the constant mean curvature equation with a
corresponding capillary boundary condition. This  is an equilibrium free surface of a fluid in a cylindrical container. When the domain $\O$ has a corner, then this nonparametric surface can be viewed as a capillary surface in a wedge (or in a wedge domain). There have been a lot of research on such a problem, especially after Concus-Finn \cite{CF69}, where it was already observed that the opening angle of the wedge and both contact angles should satisfy certain conditions for the existence. See also \cite{CF74}.
Later in \cite{CF96}  Concus-Finn proved that $f$ is continuous at a given corner if
\eqref{wedge-equiv} holds, while if \eqref{wedge-equiv} does not hold there is no solution in one case and in the left case, namely $\a> \pi -|\theta^1_0-\theta^2_0|$, they
conjectured that  $f$ has a jump discontinuity  at the corner. See the  Concus-Finn rectangle in \cite{Lan10}, Figure 2.
This conjecture was solved by Lancaster in \cite{Lan10} with the methods developed by Allard \cite{A72} and especially by Simon \cite{S80}. The latter was crucially used in a very recent work of Chao Li \cite{Li21} mentioned at the beginning of the Introduction. %where he solved a conjecture of Gromov. Condition \eqref{wedge-equiv} appears also in his work. 
See also his further work  \cite{EL22} with Edelen  on surfaces in a  polyhedral domain, which is also closely related to surfaces in a wedge domain.

The rest of the paper is organized as follows. In \cref{Sec2}, we collect some basic facts about wedges and capillary hypersurfaces in wedges.
In \cref{Sec3} and \cref{Sec4}, we prove the main theorems on the Heintze-Karcher inequality in the half-space and a wedge. In \cref{Sec5}, we prove the Alexandrov-type theorem and the non-existence result for CMC capillary hypersurfaces in a wedge.

%%%%%%%%%%%%%%%%%%%%%%%%%%%%%%%%%%%%%%%%%%%%%%%%%%%%%%%%%
\

\section{Notations and Preliminaries}\label{Sec2}

%\begin{figure}
	%\centering
	%\includegraphics[height=7cm,width=13cm]{horosphere6}
	%\caption{Hypersurface $M$  supported on  $\mathcal{H}$.}
%\end{figure}
	
 %  Let $\mfW$  be a wedge determined by $L$ hyperplanes $P_i$ ($1 
  % \le i \le L)$ and $\Sigma$ a capillary hypersurface in $\mfW$.
Let $\O$ be the bounded domain in $\mfW$  with piecewise smooth boundary $\p\O=\S\cup (\bigcup_{i=1}^LT_i)$, where  $\S=\overline{\p\O\cap\mathring{\mfW}}$ is a smooth compact embedded $\vec{\theta}$-capillary hypersurface in $\mfW$ and $T_i=\p\O\cap P_i$. 
	Denote the corners by $\Gamma_i=\S\cap T_i$, which are smooth, co-dimension two submanifolds in $\rr^{n+1}$.
	For the sake of simplicity, we denote by $\Gamma$ the union of $\Gamma_i$, i.e., $\Gamma=\bigcup_{i=1}^L\Gamma_i$.
	We use the following notation for normal vector fields.
	Let $\nu$ and $\bar N_i$ be the outward unit normal to $\S$ and $P_i$ (with respect to $\O$) respectively.
	Let $\mu_i$ be the outward unit co-normal to $\Gamma_i=\p\S\cap P_i\subset \S$ and $\bar \nu_i$ be the outward unit co-normal to $\Gamma_i\subset T_i$.
	Under this convention, along each $\Gamma_i$ %the bases
	$\{\nu,\mu_i\}$ and $\{\bar \nu_i,\bar N_i\}$ span the same 2-dimensional plane and  have the same orientation in the normal bundle of $\p \S\subset \rr^{n+1}$. 
	Hence one can define
	%In particular, $\vec{\theta}(x)$ denotes the $L$-tuple $(\theta^1(x),\ldots,\theta^L(x))$, and $\S$ is said to be $\vec{\theta}(x)$-capillary if, at $x$, along each $\Gamma_i$, 
	the contact angle function  along each $\Gamma_i$, $\theta_i:\Gamma_i\to (0,\pi)$ by
	\begin{eqnarray}\label{munu}
		&&\mu_i(x)=\sin {\theta^i(x)} \bar N_i+\cos\theta^i(x) \bar \nu_i(x),\\&& \nu(x)=-\cos \theta^i(x) \bar N_i+\sin \theta^i(x) \bar \nu_i(x).\label{defn-nu}
	\end{eqnarray} Let $\vec{\theta}(x)$ denote the $L$-tuple $(\theta^1(x), \ldots, \theta^L(x))$. We call $\Sigma$ a $
	\vec\theta$-capillary hypersurface, if we want to emphasize the contact angle function. We also use $\vec\theta_0$-capillary hypersurface to denote such a hypersurface with  $\vec\theta\equiv \vec\theta_0$, a vector $\vec\theta_0\in \prod_{i=1}^L(0,\pi).$
	
	We denote by $\bar \n$, $\bar \Delta$, $\bar \n^2$ and $\bar{\rm div}$, the gradient, the Laplacian, the Hessian and the divergence on $\rr^{n+1}$ respectively, while by $\n$, $\Delta$, $\n^2$ and ${\rm div}$, the gradient, the Laplacian, the Hessian and the divergence on the smooth part of $\p \O$, respectively. Let $g$, $h$ and $H$ be the first, second fundamental forms and the mean curvature  of the smooth part of $\p \O$ respectively. Precisely, $h(X, Y)=\<\bar\n_X\nu, Y\>$ and $H={\rm tr}_g(h)$. In particular, since $P_i$ is planar, the second fundamental form $h_i\equiv0$, correspondingly, the mean curvature $H_i$ vanishes. 
	
	We need the following structural lemma for compact hypersurfaces in $\rr^{n+1}$ with boundary, which is well-known and widely used, see \cite{AS16, JXZ22}.
	
	\begin{lemma}\label{Lem-structural}
		Let $\S\subset \rr^{n+1}$ be a smooth compact hypersurface with boundary. Then it holds that
		\begin{align}\label{EQ-structural}
			n\int_{\S}\nu \rd A=\int_{\p \S}\left\{\left<x,\mu\right>\nu-\left<x,\nu\right>\mu\right\}\rd s.
		\end{align}
	\end{lemma}
	\begin{proof}
		Let $Z=\left<\nu, e\right>x^T- \left<x,\nu\right>e^T$ for any constant vector $e$, where $x^T$ and $e^T$ denote the tangential component of $x$ and $e$ respectively. One computes that
		$${\rm div}(Z)=n\left<\nu, e\right>.$$
		Integration by parts yields the assertion.
	\end{proof}
The following lemma is well-known when the capillary hypersurfaces are bounded by containers with totally umbilical boundaries,  in particular, a wedge in $\mbR^{n+1}$, see e.g., \cite{AS16,LX17, WX19}.% We include the proof here, for the sake of completeness.

\begin{lemma}\label{Lem-principal-direction}
    Let $\mfW\subset\mbR^{n+1}$ be a wedge and $\theta^i_0\in(0,\pi)$ for $i=1,\ldots,L$.
    If $\S\subset \mfW$is a smooth $\vec{\theta}_0$-capillary hypersurface, then along $\p\S$, $\mu_i$ is a principal direction of $\S$.
\end{lemma}
\begin{proof} For the completeness we provide a proof.
    It suffices to prove that $h(\mu_i,X)=0$ for any vector $X$ tangent to $\p\S$. Indeed,
    \begin{align}
        h(\mu_i,X)
        =&\left<\bar\nabla_X\mu_i,\nu\right>
        =\left<\bar\nabla_X\left(\sin\theta^i\bar N_i+\cos\theta^i\bar\nu_i\right),-\cos\theta^i\bar N_i+\sin\theta^i\bar\nu_i\right>\notag\\
        =&\left<\bar\nabla_X\bar\nu_i,\bar N_i\right>=-h_i(\bar\nu_i,X)=0,
    \end{align}
    where we have used \eqref{munu}, the constancy of $\theta_i$, the fact that $\bar\nu_i,\bar N_i$ are unit vector fields, and $h_i=0$ since $P_i$ are totally geodesic. This completes the proof.
\end{proof}

The $r$-th mean curvature $H_r$ of $\S$ is defined by the identity:
\begin{align}\label{defn-Pn(t)}
    \mcP_n(t)=\prod_{i=1}^n(1+t\kappa_i)=\sum_{i=0}^n\binom{n}{i}H_it^i
\end{align}
for all real number $t$. Thus $H_1=\frac{H}{n}$ is the mean curvature of $\S$ and $H_n$ is the Gaussian curvature, and we adopt the convention that $H_0=1$.

We have the following Minkowski-type formula for $\vec{\theta}_0$-capillary hypersurfaces.%which is derived in \cite[Lemma 5]{LX17} as $r=1$, for the sake of completness, we include the proof here.

\begin{proposition}\label{Prop-Minkowski}
Let $\mfW\subset\mbR^{n+1}$ be a wedge and $\S\subset  \mfW$ be a $\vec{\theta}_0$-capillary hypersurface. Then it holds that for $r=1,\ldots,n$, 
%\begin{align}\label{Minkowski-wedge}
%    \int_\S n(1+\<\nu, \mfk_0\>) -H\<x,\nu\>dA=0,
%\end{align}
%where $\mfk_0$ \footnote{When $L=1$, \eqref{Minkowski-wedge} is essentially the Minkowski type formula in a half space, see e.g., \cite[Proposition 4.1]{JXZ22}.} is given in \eqref{k0}. Moreover, for $r$-th mean curvature defined in \cref{Sec2}, we have
\begin{align}\label{eq-Minkow-higherorder}
    \int_\S\left(H_{r-1}\left(1+\left<\nu,\mfk_0\right>\right)-H_r\left<x,\nu\right>\right)\rd A=0.
\end{align}
In particular, if $L=1$, i.e. $\S\subset \bar{\mathbb{R}}^{n+1}_+$, then
\begin{eqnarray}\label{mink-formula} \int_\S H_{r-1}(1-\cos\theta_0\<\nu, E_{n+1}\>)-H_r\<x,\nu\> \rd A=0. \end{eqnarray} 
\end{proposition}

\begin{proof}The case for $r=1$ has been proved in \cite[Lemma 5]{LX17}. For the sake of completeness, we include the proof here.

		Since $${\rm div}(x^T)=n-H\left<x, \nu\right>,$$ by using integration by parts in $\S$, we get
		\begin{align}\label{xxeq1}
		\int_\S (n -H\<x,\nu\>)\rd A
		=\sum_{i=1}^L\int_{\Gamma_i}\<x, \mu_i\> \rd s_i.
		\end{align} 
		From the capillary boundary condition \eqref{munu} it is easy to see
		that on each $\Gamma_i$
	%	\begin{align}
		%    -\cos\theta^i_0\nu+\sin\theta^i_0\mu_i=\bar N_i,
	%	\end{align}
	%	if follows that
		\begin{align}\label{xmu}
		    -\cos\theta^i_0\left<x,\nu\right>+\sin\theta^i_0\left<x,\mu_i\right>=\left<x,\bar N_i\right>=0.
		\end{align}
		By \eqref{EQ-structural}, \eqref{munu} and \eqref{xmu}, we get
		\begin{align}\label{xxeq2}
		n\int_\S \<\nu, \mfk_0\> \rd A
		&=\sum_{i=1}^L\int_{\Gamma_i}\left(\left<x,\mu_i\right>\<\nu, \mfk_0\>-\left<x,\nu\right>\<\mu_i,\mfk_0\>\right) \rd s_i\nonumber	\\
		&=\sum_{i=1}^L\int_{\Gamma_i}\left(\left<x,\mu_i\right>\left<\nu,\mfk_0\right>
		+\left<x,\nu\right>\frac{-\cos\theta^i_0}{\sin\theta^i_0}\left(1+\left<\nu,\mfk_0\right>\right)\right)\rd s_i\notag\\
		&=\sum_{i=1}^L\int_{\Gamma_i}\left(\left<x,\mu_i\right>\left<\nu,\mfk_0\right>
		-\left<x,\mu_i\right>\left(1+\left<\nu,\mfk_0\right>\right)\right)\rd s_i\notag\\
		&=-\sum_{i=1}^L\int_{\Gamma_i}\left<x,\mu_i\right>\rd s_i,
		\end{align}
		%where we have used the fact that $\left<x,\bar N_i\right>=0$ for any $x\in P_i$ in the last equality, since the origin $\mathbf{O}\in\bigcap_{i=1}^LP_i$.
		It follows from \eqref{xxeq1} and \eqref{xxeq2} that
		\begin{align}\label{Minkowski-wedge}
 \int_\S n(1+\<\nu, \mfk_0\>) -H\<x,\nu\>\rd A=0.
\end{align}
		
		Now we prove \eqref{eq-Minkow-higherorder} for general $r$. For a small real number $t>0$, consider a family of hypersurfaces with boundary $\S_t$, % the outwards capillary parallel hypersurface, 
		 defined by 
		$$y:=\varphi_t(x)=x+t(\nu(x)+\mfk_0)\quad x\in\S.$$We claim that $\S_t$ is also a $\vec\theta_0$-capillary hypersurface in $\mfW$.
		In fact, if $e_1,\ldots,e_n$ are principal directions of a point of $\S$ and $\kappa_i$ are the corresponding principal curvatures, we have
		\begin{align}\label{eq-principaldirections}
		    (\varphi_t)_\ast(e_i)=\bar\nabla_{e_i}\varphi_t=(1+t\kappa_i)e_i,\quad i=1,\ldots,n.
	\end{align}
	From \eqref{eq-principaldirections}, we see that $\nu_t(y)=\nu(x)$, where $\nu_t(y)$ denotes the outward unit normal of $\S_t$ at $y=\varphi_t(x)$.
	Moreover, the capillarity condition \eqref{defn-nu} implies: for any $x\in\p\S\cap P_i$, we have
	\begin{align}
	    \left<\nu(x)+\mfk_0,\bar N_i\right>=-\cos\theta_0^i+\cos\theta_0^i=0,
	\end{align}
	in other words, $\varphi_t(x)\in P_i$, and hence $\p\S_t\subset\varphi_t(\p\S)$. In view of this, we have: $\left<\nu_t(y),\bar N_i\right>=\left<\nu(x),\bar N_i\right>=-\cos\theta^i_0$; that is, $\S_t$ is also a $\vec\theta_0$-capillary hypersurface in $\mfW$.
	
	Therefore, we can exploit \eqref{Minkowski-wedge} to find that
	\begin{align}\label{eq-PropMin-0} \int_{\S_t=\varphi_t(\S)}n\left(1+\left<\nu_t,\mfk_0\right>\right)-H(t)\mid_y\left<y,\nu_t\right>\rd A_t(y)=0.
	\end{align}
%Notice that the parallel translation implies $\nu_t(y)=\nu(x)$, and it follows that
%$\left<y,\nu_t(y)\right>=\left<x,\nu\right>+t\left(1+\left<\nu(x),\mfk_0\right>\right)$. For $x\in\S$, let $e_i$ be principal unit vector of $\S$ at $x$ and $\kappa_i(x)$ be the corresponding principal curvatures, a simple computation then yields, for $i=1,\ldots,n$,
%\begin{align*}
%    \bar\nabla_{e_i}\varphi_t(x)=(1+t\kappa_i(x))e_i,
%\end{align*}
By \eqref{eq-principaldirections}, the tangential Jacobian of $\varphi_t$ along $\S$ at $x$ is just
\begin{align}\label{eq-PropMin-1}
    {\rm J}^\S\varphi_t(x)=\prod_{i=1}^n(1+t\kappa_i(x))=\mcP_n(t),
\end{align}
where $\mcP_n(t)$ is the polynomial defined in \eqref{defn-Pn(t)}.
Moreover, using \eqref{eq-principaldirections} again, we see that
%a classical computation shows that, if $e_1,\ldots,e_n$ are principal directions at $x\in\S$, then they are also the principal directions at $\varphi_t(x)\in\S_t$, and
the corresponding principal curvatures are given by
\begin{align}\label{eq-PropMin-2}
    \kappa_i(\varphi_t(x))=\frac{\kappa_i(x)}{1+t\kappa_i(x)}.
\end{align}
Hence fix $x\in\S$, the mean curvature of $\S_t$ at  $\varphi_t(x)$, say $H(t)$, is given by
\begin{align}
\label{29}
    H(t)=\frac{\mcP'_n(t)}{\mcP_n(t)}=\frac{\sum_{i=0}^{n}i\binom{n}{i} H_it^{i-1}}{\mcP_n(t)},
\end{align}
where $H_i=H_i(x)$ is the $i$-th mean curvature of $\S$ at $x$.

Using the area formula, \eqref{eq-PropMin-1} and \eqref{29}, we find from \eqref{eq-PropMin-0} that
\begin{align}
    \int_{\S}n\left(1+\left<\nu,\mfk_0\right>\right)\mcP_n(t)-t\left(1+\left<\nu,\mfk_0\right>\right)\mcP'_n(t)-\mcP'_n(t)\left<x,\nu\right>\rd A_x=0.
\end{align}
As the left hand side in this equality is a polynomial in the time variable $t$, this shows that all its coefficients vanish, and hence %and these coefficients are, up to a constant multiple, indeed
\begin{align}
    \int_\S\left(1+\left<\nu,\mfk_0\right>\right)H_{r-1}-H_r\left<x,\nu\right>\rd A,\quad r=1,\ldots,n.
\end{align}
This gives \eqref{eq-Minkow-higherorder}.
\end{proof}
We have also
\begin{align}\label{mink0}
        (n+1)\vert\O\vert=\int_\S\left<x,\nu\right>\rd A,
    \end{align}
    which is easy to prove. If one views \eqref{mink0} as one of \eqref{eq-Minkow-higherorder} with $r=0$, the Heintze-Karcher inequality \eqref{EQ-HK-halfspace} that we want to  prove could also be viewed one of them with $r=-1$. Certainly now it is an inequality, instead of an equality.

\begin{remark}
\normalfont
An alternative proof of \cref{Prop-Minkowski} can be given as that of \cite[Proposition 2.5]{WWX22}, where the Minkowski-type formula for the half-space case has been proved.
\end{remark}

\ 

In the sequel, $\Sigma$ will be always referred to as a smooth, embedded capillary hypersurface.
%%%%%%%%%%%%%%%%%%%%%%%%%%%%%%%%%%%%%%%%%%%%%%%%%%%%%%%%%

%\section{Heintze-Karcher inequality in a half space}\label{Sec3}

%Let $x: \S\to \R^{n+1}_+$ be an compact embedded hypersurface with $\theta$-capillary boundary. Let 
%$$x(t, p)=x(p)+t(\nu(p)+\cos\theta e),$$
%where $e=-E_{n+1}$.

%Let $e_i$ be principal unit vector of $x(\S)$ at $p$ and $\kappa_i$ be the corresponding principal curvature, then
%\begin{align} 
%\p_t x(t, p)=\nu(p)+\cos\theta e, \\ 
%\bar \n_{e_i} x(t, p)= (1-t\kappa_i )e_i.
%\end{align}
%Hence 
%$${\rm Jac}(x)(t, p)= (1+\cos\theta\<\nu, e\>)\prod_{i=1}^n(1-t\kappa_i).$$
%Define a subset of $\R\times \S$
%\begin{align}
%S=\left\{(t, p)\in \R\times \S \Big| 0\le t\le \frac{1}{\max_i \kappa_i(p)}\right\}.
%\end{align}
%Since $H>0$, we know $\frac{1}{\max_i \kappa_i(p)}>\frac{1}{H(p)}$ and hence $S$ is non-empty.

%Let $a\in \R$ such that 
%$$a\ge \max_{p\in \S}\frac{1}{\max_i \kappa_i(p)},$$ 
%then $S$ is a compact set in $[0, a)\times \S$.

%Applying the area formula and we find
% \begin{align}
%&|\Omega|\le \int_\S\int_0^a \chi_S(t, p)|{\rm Jac}(x)(t, p)|dtdA\\
%&\le \int_\S\int_0^{\frac{1}{\max_i \kappa_i(p)}} (1+\cos\theta\<\nu, e\>)\prod_{i=1}^n(1-t\kappa_i)dtdA\\
%&\le \int_\S\int_0^{\frac{1}{H(p)}} (1+\cos\theta\<\nu, %e\>)(1-t\frac{H}{n})^n dtdA\\
%&=\int_\S \frac{n(1+\cos\theta\<\nu, e\>)}{(n+1)H} dA.
% \end{align}

\

\section{Heintze-Karcher Inequality in the Half-Space}\label{Sec3}

\begin{proof}[Proof of \cref{Thm-HK-halfspace}]
Let $\S\subset\overline{\mbR^{n+1}_+}$ be a $
\theta$-capillary hypersurface with  $\theta\leq\theta_0$ along $\p\S$.
For any $x\in \S$, let $\{e_i=e_i(x)\}$ be the set of  unit  principal  vectors of $\S$ at $x$ and $\{\kappa_i(x)\}$  the set of  corresponding principal curvatures. Since $\S$ is strictly mean convex, $$\max_i{\kappa_i(x)}\ge \frac{H(x)}{n}>0, \hbox{ for }x\in \S.$$
We define 
\begin{align*}
    Z=\left\{(x,t)\in\S\times\mbR:0<t\leq\frac{1}{\max{\kappa_i(x)}}\right\},
   \end{align*} 
   and
   \begin{align}
    &\zeta: Z\to \rr^{n+1},\nonumber\\
    &\zeta(x,t)=x-t\left(\nu(x)-\cos\theta_{0} E_{n+1}\right).\label{parallel}
\end{align}
$\zeta$ gives a family of hypersurfaces  $\zeta (\Sigma, t)$, which are the modified parallel hypersurfaces mentioned above.  

{\bf Claim:} $\O\subset\zeta(Z)$.

Indeed, let us denote by $B_r(x)$ the closed  ball centered at $x$ of radius $r$, and $S_r(x)=\partial B_r(x)$.
For any $y\in\O$, we consider a family of spheres $\{S_r(y-r\cos\theta_{0} E_{n+1})\}_{r\geq0}$.
Since $y\in\O$ is an interior point, when $r$ is small enough, we  have $B_r(y-r\cos\theta_{0} E_{n+1})\subset\subset\O$. Since $\vert\cos\theta_{0}\vert<1$, it is easy to see that the spheres gives a foliation of $\mbR^{n+1}$.
Hence %$B_r(y-r\cos\theta_{0} E_{n+1})$ 
$S_r(y-r\cos\theta_{0} E_{n+1})$
must touch $\S$ as we increase the radius $r$. As a conclusion, for any $y\in\O$, there exists $x\in\S$ and $r_y>0$, such that $S_r(y-r_y\cos\theta_{0} E_{n+1})$
touches $\S$ for the first time, at the point $x\in \S$.
We have two cases.

\textbf{Case 1.} $x\in\mathring{\Sigma}$. 

In this case, since $x\in\mathring{\S}$, 
%the fact that % $B_{r_y}(y-r_y\cos\theta_{0} E_{n+1})$ 
%touches $\S$ for the first time then implies:
the sphere $S_r(y-r_y\cos\theta_{0} E_{n+1})$ is tangent to $\S$ at $x$  from the interior.  It follows  that $r_y\leq\frac{1}{\max{\kappa_i(x)}}$. Invoking the definition of $Z$ and $\zeta$, we find that $y\in\zeta(Z)$ in this case.

 \textbf{Case 2.} $x\in\p\S$. 

We will rule out this case by the condition on the contact angle function $\theta$ of $\S$.  In this case, %$B_{r_y}(y-r_y\cos\theta_0 E_{n+1})
%$S_r(y-r_y\cos\theta_{0} E_{n+1})\cap \p \mbR^{n+1}_+$ is tangent to $\S\cap \p \mbR^{n+1}_+$ at $x$. One can easily see that 
by the  first touching property of $x$, the contact angle  $\theta_y$ of $S_r(y-r_y\cos\theta_{0} E_{n+1})$ with $\p \mbR^{n+1}_+$ is smaller than or equals to $\theta(x)$, which is smaller than or equals to $\theta_0$, by assumption.
%
%and the touching angle $\bar\theta(x)$ of $B_{r_y}(y-r_y\cos\theta_0 E_{n+1})$ with $\p \mbR^{n+1}_+$ at $x$ must be smaller than or equal to $\theta(x)$, and it follows from the geometric relation that $y-\left(r_y\cos\theta_0-r_y\cos{\bar\theta(x)}\right)E_{n+1}\in\p \mbR^{n+1}_+$
(see \cref{Fig-1} for an illustration). However $\theta_y \le \theta_0$ implies that $\langle y, E_{n+1}\rangle <0$, a contradiction to $y\in \O\subset\mbR^{n+1}_+$. 
%Since $\bar\theta(x)\leq\theta(x)\leq\theta_0<\pi$, we know it must be that $\<y, E_{n+1}\>\le 0$, which contradicts to the fact that $y\in\O\subset\mbR_+^{n+1}$.
The {\bf Claim} is thus proved.
\begin{figure}[H]
	\centering
	\includegraphics[width=10cm]{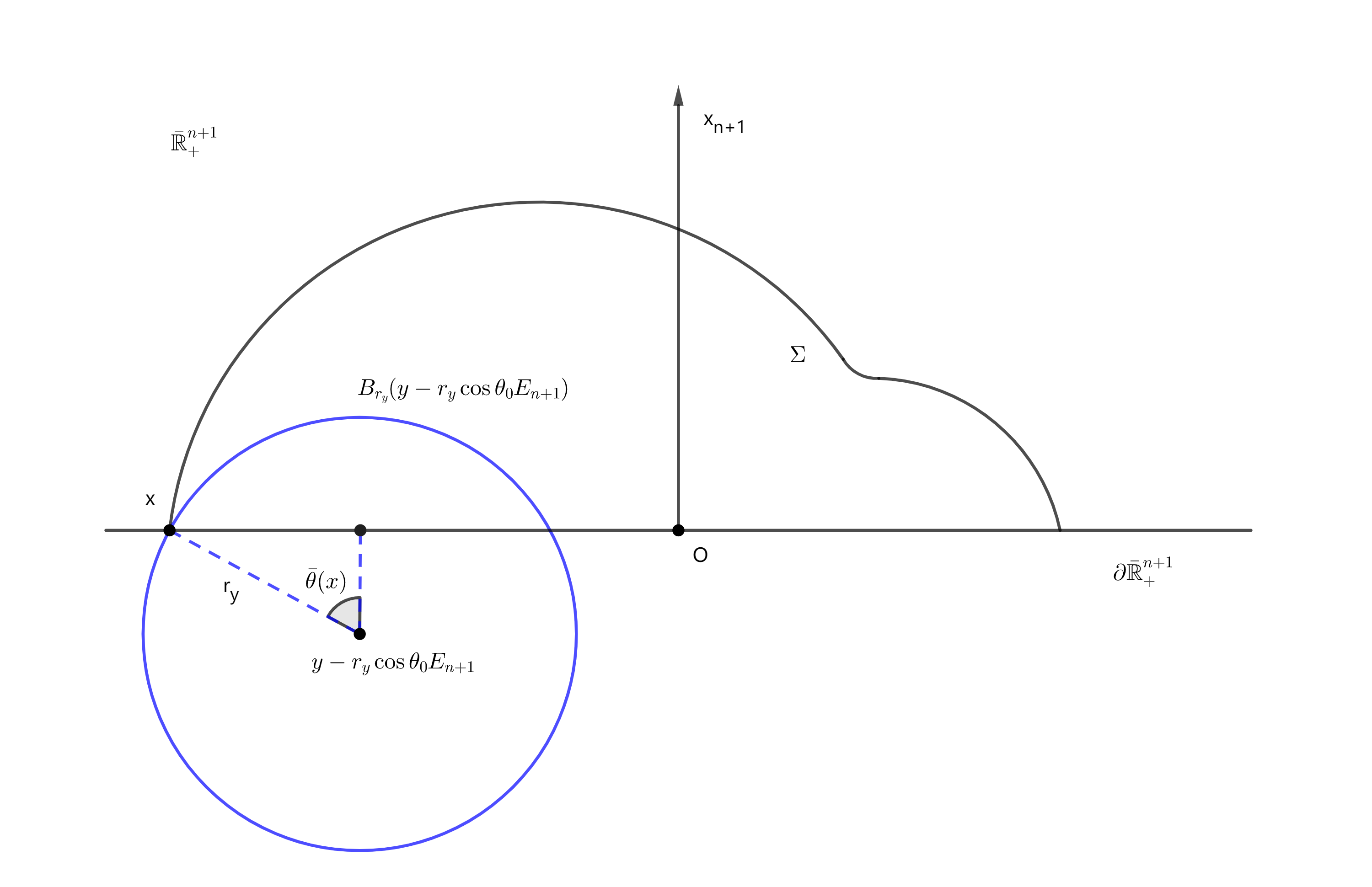}
	\caption{Boundary touching.}
	\label{Fig-1}
\end{figure}
 
By a simple computation, we find
\begin{align*} 
\p_t \zeta(x, t)&=-\left(\nu(x)-\cos\theta_0 E_{n+1}\right), \\ 
\bar \n_{e_i} \zeta(x, t)&=\left(1-t\kappa_i(x)\right)e_i.\label{eq-Thm1-2}
\end{align*}
Hence the tangential Jacobian of $\zeta$ along $Z$, at $(x,t)$ is just
$${\rm J}^Z\zeta(x,t)= (1-\cos\theta_0\<\nu, E_{n+1}\>)\prod_{i=1}^n(1-t\kappa_i).$$
By virtue of the fact that $\O\subset\zeta(Z)$, the area formula yields
\begin{align*}
    \vert\O\vert\leq\vert\zeta(Z)\vert
    \leq&\int_{\zeta(Z)}\mcH^0(\zeta^{-1}(y))\rd y
    =\int_Z{\rm J}^Z\zeta \rd\mcH^{n+1}\notag\\
    =&\int_\Sigma \rd A\int_0^{\frac{1}{\max\left\{\kappa_i(x)\right\}}}\left(1-\cos\theta_0\left<\nu,E_{n+1}\right>\right)\prod_{i=1}^n(1-t\kappa_i(x))\rd t.
\end{align*}
By the AM-GM inequality, $1-\cos\theta_0\left<\nu,E_{n+1}\right>>0$ on $\S$, and the fact that $\max\left\{\kappa_i(x)\right\}\geq H(x)/n$, we obtain
\begin{align*}
    \vert\O\vert
    \leq&\int_\S \rd A\int_0^{\frac{1}{\max\left\{\kappa_i(x)\right\}}} \left(1-\cos\theta_0\<\nu, E_{n+1}\>\right)\left(\frac{1}{n}\sum_{i=1}^n\left(1-t\kappa_i(x)\right)\right)^n \rd t\notag\\
    \leq&\int_\S \left(1-\cos\theta_0\<\nu, E_{n+1}\>\right)\rd A\int_0^{\frac{n}{H(x)}} \left(1-t\frac{H(x)}{n}\right)^n \rd t\notag\\
    =&\frac{n}{n+1}\int_\S \frac{(1-\cos\theta_0\<\nu, E_{n+1}\>)}{H} \rd A,
\end{align*}
which is \eqref{EQ-HK-halfspace}. 

The characterization of equality case in \eqref{EQ-HK-halfspace} follows from the classical one. Precisely, since the equality holds throughout the argument, the arithmetic mean-geometric mean (AM-GM) inequality assures the umbilicity of $\S$, and it follows that $\S$ is a spherical cap.  Apparently,
the contact angle of a spherical cap with a hyperplane is a constant, say $\theta$. It is easy to see that $\theta=\theta_0$,
%We can then test the Heintze-Karcher inequality \eqref{EQ-HK-halfspace} with this $\theta$-cap to find that the equality in \eqref{EQ-HK-halfspace} holds only if $\theta=\theta_0$ (see \cref{Rem-thetacaps} for details).
%In conclusion,
and hence $\S$ must be a $\theta_0$-capillary spherical cap. 
Conversely, when $\S$ is a $\theta_0$-capillary spherical cap, then $H$ is a positive constant. By virtue of the Minkowski formula \eqref{mink-formula} for $r=1$, we see that equality in \eqref{EQ-HK-halfspace} holds.
\end{proof}

Let us close this section with a remark. In the proof of $\Omega\subset\zeta(Z)$, our choice of the touching balls is enlightened by the following observation. 
\begin{remark}[Foliation of $\theta_0$-capillary hypersurfaces]\label{Rem-foliation}
\normalfont In \cite{MR91}, to prove the Heintze-Karcher inequality for closed hypersurfaces, one shall `sweepout' the domain $\Omega$ by a foliation around any point $p\in\Omega$, whose leaves are level-sets of the distance function to $p$. The key point is that, such `sweep-outs' coincides with the domain $\Omega$, if and only if $\Omega$ is a ball and $p$ is chosen to be the center.

In view of this, our choice of foliation in the capillary case is thus clear; we want to `sweepout' the $\theta_0$-ball with the foliation, whose leaves are $\theta_0$-spherical caps(as illustrated in \cref{Fig-capillaryfoliation}). 
\end{remark}

\begin{figure}[htbp]
    \begin{minipage}[t]{0.5\linewidth}
        \centering
        \includegraphics[width=\textwidth]{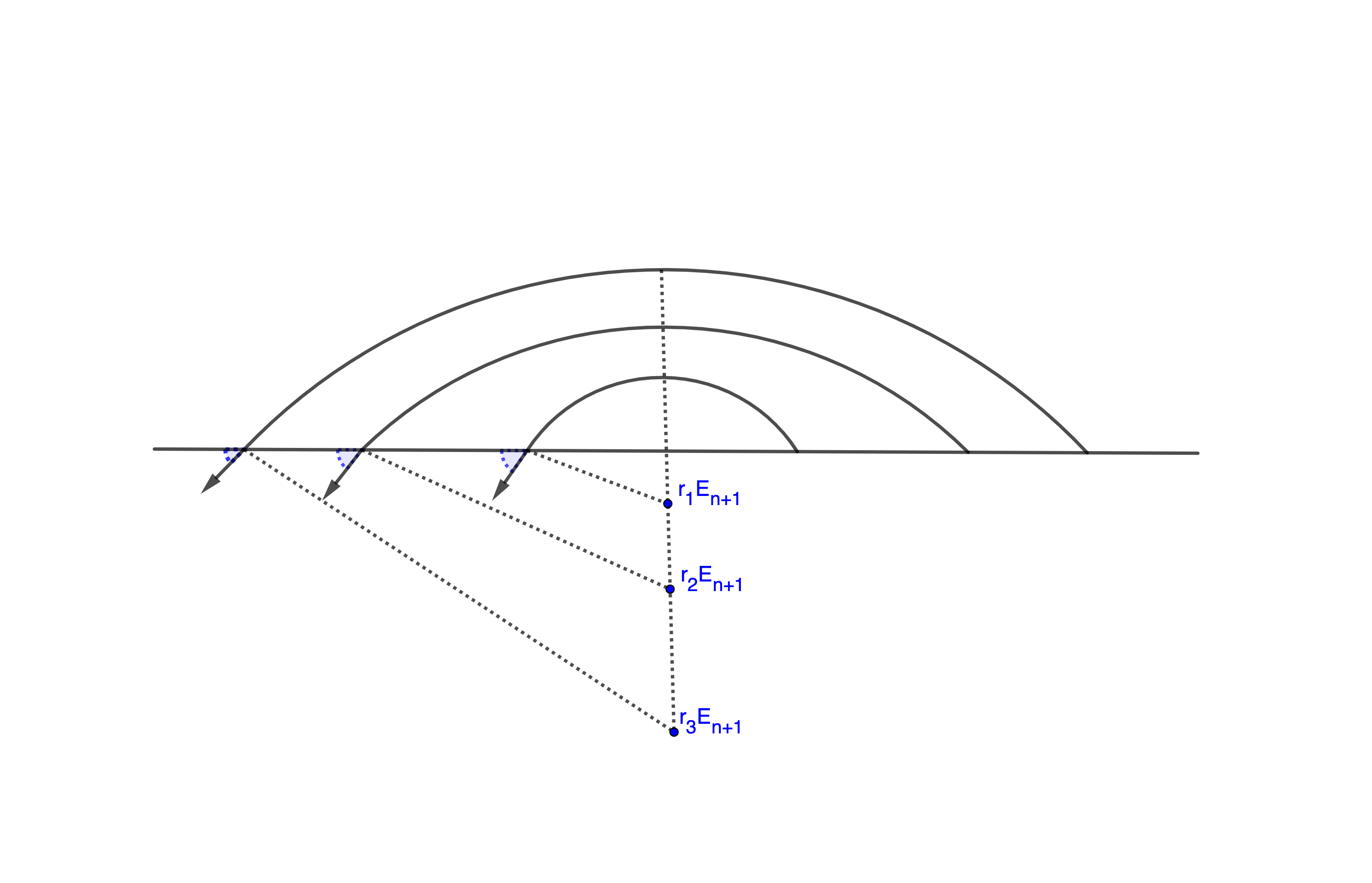}
        \centerline{(a) Sweepout of $\theta_0$-domain.}
    \end{minipage}%
    \begin{minipage}[t]{0.5\linewidth}
        \centering
        \includegraphics[width=\textwidth]{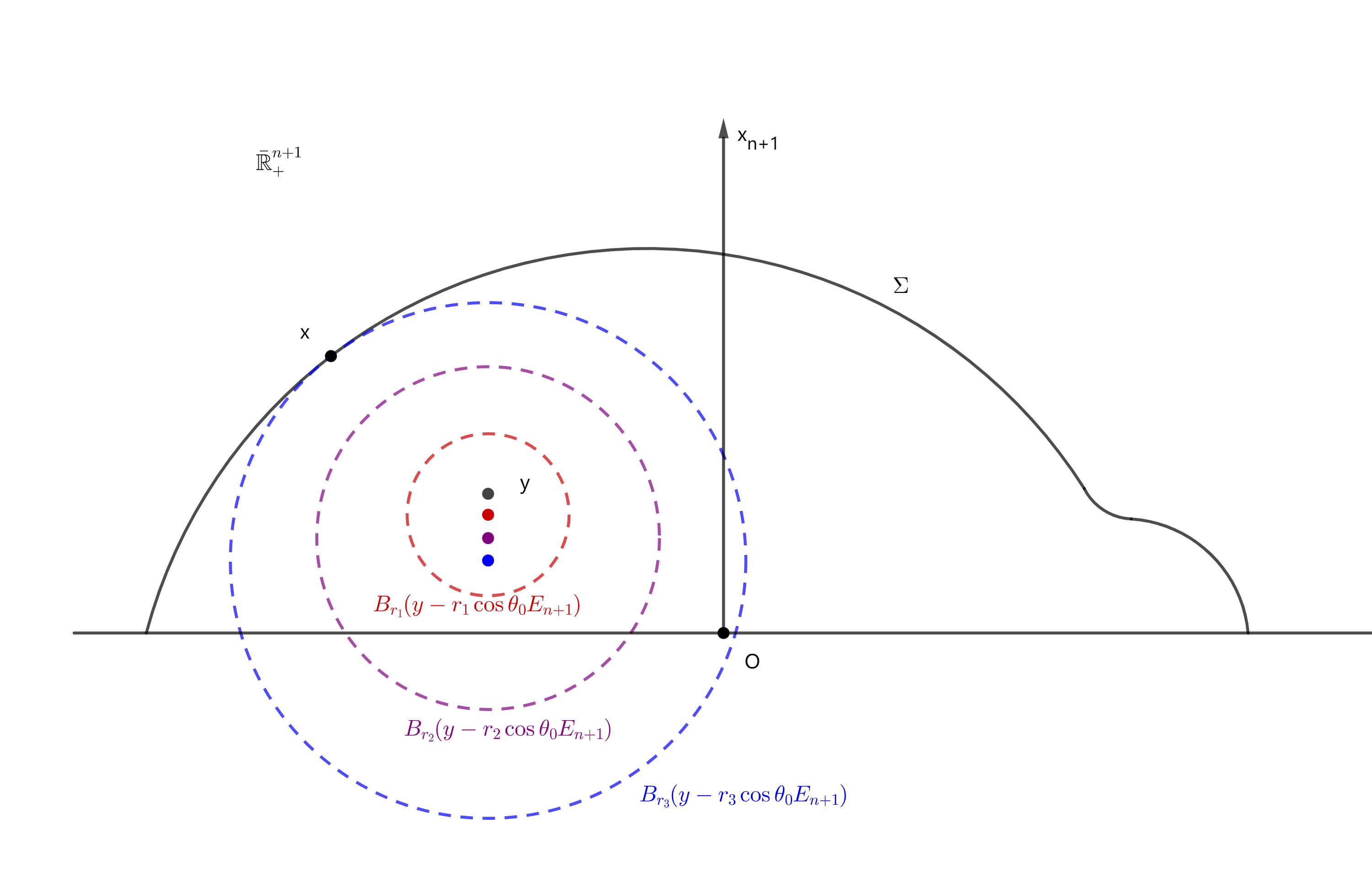}
        \centerline{(b) Sweepout of generic domain.}
    \end{minipage}
    \caption{Capillary foliation.}
    \label{Fig-capillaryfoliation}
\end{figure}

\

\section{Heintze-Karcher Inequality in a Wedge}\label{Sec4}
In this section, we prove \cref{Thm-HK-wedge} and \cref{Thm-HK-wedge-2} in a wedge $\mfW$, following largely from the proof of the Heintze-Karcher-type inequality presented  in the previous section.
\begin{proof}[Proof of \cref{Thm-HK-wedge}]
Let $\S\subset \mfW$ be a compact embedded hypersurface with $\vec{\theta}(x)$-capillary boundary,
where $\theta^i(x)\leq\theta^i_0$ for each $i$ and every $x\in\Gamma$. As above, 
for any $x\in \S$, let $\{e_i\}$ be the set of  principal unit vectors of $\S$ at $x$ and $\{\kappa_i(x)\}$ the set of the corresponding principal curvatures.
Now we  define modified parallel hypersurfaces by 
\begin{align*}
    &Z=\left\{(x,t)\in\S\times\mbR:0<t\leq\frac{1}{\max{\kappa_i(x)}}\right\},\\
    &\zeta(x,t)=x-t\left(\nu(x)+\mfk_0\right),\text{ }(x,t)\in Z.
\end{align*}
%where $\mfk_0=\sum_{i=1}^{L}c_i\bar{N_i}$, satisfying $\left<\mfk_0,\bar N_i\right>=\cos\theta^i_0$ for each $i$.

As in \cref{Thm-HK-halfspace}, we shall show that $\O\subset\zeta(Z)$. Let $y\in \O$. We consider the sphere foliation $\{S_r(y+r\mfk_0)\}_{r\geq0}$. By virtue of \cref{Lem-k0<=1-1}, there exists some $r_y>0$ such that $S_{r_y}(y+r_y\mfk_0)$ touches $\S$ from the interior at a first touching point $x\in\S$. 

\textbf{Case 1.} $x\in\mathring{\S}.$  We can get $y\in\zeta(Z)$, as argued in {\bf Case 1}, \cref{Thm-HK-halfspace}. 

%Indeed, let us denote by $B_r(x)$ the Euclidean ball in $\mbR^{n+1}$ with center at $x$ and radius $\rho$. For any $y\in\O$, we consider $B_r(y+r\mfk)$. Notice that since $y\in\O$ is an interior point, when $r$ is small enough, we definitely have $B_r(y+r\mfk)\subset\subset\O$. On the other hand, since $\vert\mfk\vert<1$ strictly, we know that $B_r(y+r\mfk)$ must touch $\bar\S$ as we increase the radius $r$. Thus, for any $y\in\O$, there exists $x\in\bar\S$ and $r_x>0$, such that $B_{r_x}(y+r_x\mfk)$ touches $\bar\S$ for the first time, at the point $x$.

%\textbf{Case1.} $x\in\mathring{\Sigma}$. 

%In this case, since $x\in\mathring{\S}$, the fact that $B_{r_x}(y+r_x\mfk)$ touches $\bar\S$ for the first time then implies: this ball is tangent to $\S$ at $x$ {\color{red} from the interior}. Since $\S$ is mean convex, we have $\max\kappa_i>0$.  It follows  that $r_x\leq\frac{1}{\max{\kappa_i(x)}}$. Revoking the definition of $Z$ and $\zeta$, we find that $y\in\zeta(Z)$ in this case.
\textbf{Case 2.} $x\in\p\S$. By assumption \eqref{away-edge},  $x\in\p\S\cap \mathring{P}_i$ for some $i$.
We will rule out this case again by virtue of the capillarity of $\S$. 
 In this case, $S_{r_y}(y+r_y\mfk_0)\cap P_i$ is tangent to $\S\cap P_i$, and the touching angle of $S_{r_y}(y+r_y\mfk_0)$ with $P_i$ at $x\in\Gamma_i$ must be smaller than $\theta^i(x)$, and it follows from the geometric relation that $y$ lies outside $\mfW$. Precisely, up to a rotation, we may assume that the touching plane is $\left\{x_{n+1}=0\right\}$, say $P_1$, and we denote by $\bar\theta^1(x)$ the touching angle, satisfying $\bar\theta^1(x)\leq\theta^1(x)$, due to the first touch.
From the geometric relation(see \cref{Fig-firsttouching-wedge}), we find, $y+r_y\left(\mfk_0-\cos\bar\theta^1(x)\bar N_1\right)\in P_1$, the angle relation $\bar\theta^1(x)\leq\theta^1(x)\leq\theta^1_0<\pi$ then implies that $y\in\mbR^{n+1}\setminus\mfW$ \footnote{Notice that $\left<\mfk_0,\bar N_1\right>=\cos\theta_0^1$, which means moving along $\mfk_0$ with distance $r_y$ is indeed moving along $\bar N_1$ with distance $r_y\cos\theta^1$.}, which contradicts to the fact that $y\in\O\subset\mfW$. 
Therefore, we complete the proof that $\O\subset\zeta(Z)$.
\begin{figure}
	\centering
	\includegraphics[width=10cm]{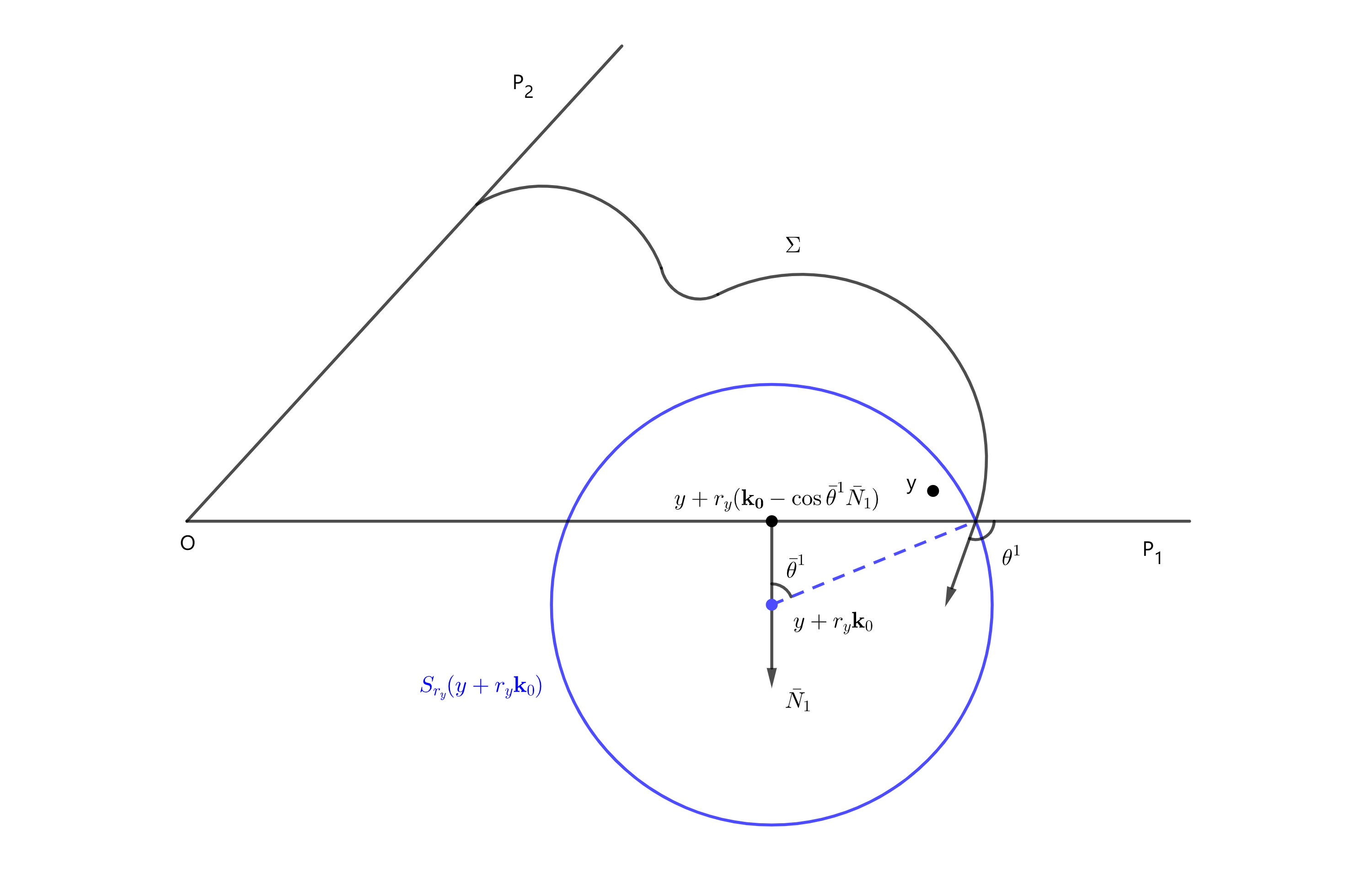}
	\caption{Touching supporting hyperplanes in the interior.}
	\label{Fig-firsttouching-wedge}
\end{figure}

%By a simple computation, we find
%\begin{align} 
%\p_t \zeta(x, t)&=-\left(\nu(x)+\mfk_0\right), \\ 
%\bar \n_{e_i} \zeta(x, t)&=\left(1-t\kappa_i(x)\right)e_i.\label{eq-Thm1-2}
%\end{align}

By a simple computation as \cref{Thm-HK-halfspace}, we see, the tangential Jacobian of $\zeta$ along $Z$ at $(x,t)$ is just
$${\rm J}^Z\zeta(x,t)= (1+\<\nu, \mfk_0\>)\prod_{i=1}^n(1-t\kappa_i).$$
By a similar argument as \cref{Thm-HK-halfspace}, we conclude
\begin{align}
 \vert\O\vert\le\frac{n}{n+1}\int_\S \frac{(1+\<\nu, \mfk_0\>)}{H}\rd A.
\end{align}
%By virtue of the fact that $\O\subset\zeta(Z)$, the area formula, the AM-GM inequality, $\vert\mfk_0\vert\leq1$, and the fact that $\max\left\{\kappa_i(x)\right\}\geq H(x)/n$, we arrive at
%\begin{align}
%    \vert\O\vert\leq\vert\zeta(Z)\vert
%    \leq&\int_{\zeta(Z)}\mcH^0(\zeta^{-1}(y))dy
 %   =\int_Z{\rm J}^Z\zeta d\mcH^{n+1}\notag\\
 %   =&\int_\Sigma dA_x\int_0^{\frac{1}{\max\left\{\kappa_i(x)\right\}}}\left(1+\left<\nu,\mfk_0\right>\right)\prod_{i=1}^n(1-t\kappa_i(x))dt\notag\\
 %   \leq&\int_\S dA_x\int_0^{\frac{1}{\max\left\{\kappa_i(x)\right\}}} \left(1+\<\nu, \mfk_0\>\right)\left(\frac{1}{n}\sum_{i=1}^n\left(1-t\kappa_i(x)\right)\right)^n dt\notag\\
 %   \leq&\int_\S \left(1+\<\nu, \mfk_0\>\right)dA_x\int_0^{\frac{n}{H(x)}} \left(1-t\frac{H(x)}{n}\right)^n dt\notag\\
 %   =&\frac{n}{n+1}\int_\S \frac{(1+\<\nu, \mfk_0\>)}{H} dA,
%\end{align}
%\begin{align}
%    \vert\O\vert
%    \leq&\int_\S dA_x\int_0^{\frac{1}{\max\left\{\kappa_i(x)\right\}}} \left(1+\<\nu, \mfk\>\right)\left(\frac{1}{n}\sum_{i=1}^n\left(1-t\kappa_i(x)\right)\right)^n dt\notag\\
%    \leq&\int_\S \left(1+\<\nu, \mfk\>\right)dA_x\int_0^{\frac{n}{H(x)}} \left(1-t\frac{H(x)}{n}\right)^n dt
%    =\frac{n}{n+1}\int_\S \frac{(1+\<\nu, \mfk\>)}{H} dA,
%\end{align}
%which is \eqref{EQ-HK-wedge}. 

As proved in \cref{Thm-HK-halfspace}, if equality in \eqref{EQ-HK-wedge} holds, then $\S$ is umbilical, and hence spherical. To see that $\S$ must be a $\vec\theta_0$-capillary spherical cap, we need a different argument.
As equalities hold throughout the argument, we have
\begin{align}\label{eq-Omega-Zeta}
    \vert\Omega\vert=\vert\zeta(Z)\vert=\int_{\zeta(Z)}\mcH^0(\zeta^{-1}(y))\rd y.
\end{align}
Moreover, \eqref{defn-nu} implies: for any $x\in\p\S\cap P_i$, there holds
	\begin{align}
	    \left<-(\nu(x)+\mfk_0),\bar N_i\right>=\cos\theta^i(x)-\cos\theta^i_0\geq0.
	\end{align}  Recall that $\bar N_i$ is the outwards pointing unit normal of $P_i$, and we have already showed in the previous proof that $\Omega\subset\zeta(Z)$. Thus, if $\theta^i(x)<\theta^i_0$ strictly at some $x\in\S\cap P_i$, then it must be that $\vert\Omega\vert<\vert\zeta(Z)\vert$, which contradicts to \eqref{eq-Omega-Zeta}.
	In other words, for any $x\in\p\Sigma\cap P_i$, we must have $\theta^i(x)=\theta^i_0$, this shows that $\S$ must be a $\vec\theta_0$-capillary spherical cap.
%The inverse characterization follow again by Minkowski-type formula \eqref{eq-Minkow-higherorder}. This completes the proof.
\end{proof}

\begin{proof}[Proof of \cref{Thm-HK-wedge-2}]
We note that the proof follows closely  the one of \cref{Thm-HK-wedge}. 
Precisely, thanks to \cref{Lem-k0<=1-1}, we can use our foliation $\{S_r(y+r\mfk_0)\}_{r\geq0}$ to test the surjectivity of $\zeta$, i.e., $\Omega\subset\zeta(Z)$. One subtle point 
 we have to be concerned with is 
that the first touching of $S_{r_y}(y+r_y\mfk_0)$ with $\S$ might occur at $\S\cap P_1\cap P_2$.
%the surjectivity of $\zeta$, i.e., $\Omega\subset\zeta$. Precisely, it is possible that $B_{r_y}(y+r_y\mfk_0)$ touches $\S$ from the interior at a first touching point  $x\in\p\S\cap P_1\cap P_2$.
Here we manage to  rule   this case out by a rather subtle analysis.
%{\bf Case3. }$x\in\p\S\cap P_1\cap P_2$.
{In view of \cref{higher-dim} below, We only need to consider the 3-dimensional case, i.e., $n+1=3$, in which case $P_1\cap P_2$ is a line.}

In the following we use $\nu_{B_r}(x)$ to denote the outward unit normal  of $S_{r_y}(y+r_y\mfk_0)$, $T_x \Sigma$ to denote the tangent plane of $\Sigma$ at $x$, and  $l$ to denote a unit vector generating the line  $P_1\cap P_2$. %Let $S_{r_y}(y+r_y\mfk_0)=\partial B_{r_y}(y+r_y\mfk_0)$. 
Recall that $\bar N_i$ is the outward unit normal of $P_i$, $i=1, 2.$ 
%In the following, we will consider the situations that
%there exists a normal of $P_i$ which is parallel to $\nu_{B_r}$, which can be excluded by the property of balls.

\

{\bf Case 1.} $\bar N_i\parallel\nu_{B_r}(x)$ for some $i$.

Without loss of generality, we assume $\bar{N_2}$
 is parallel to $\nu_{B_r}(x)$. Hence the sphere $S_{r_y}(y+r_y\mfk_0)$ touches the plane $P_2$  only at the point $x$. Since $\bar{N_1}$ and $\bar{N_2}$ are not parallel, then $\bar{N_1}$
 is not parallel to $\nu_{B_r}$,
 thus the intersection of $S_{r_y}(y+r_y\mfk_0)$ with $P_1$ must be a circle. Since $S_{r_y}(y+r_y\mfk_0)\cap P_2=\{x\}$, we know that  $S_{r_y}(y+r_y\mfk_0)\cap P_1$ touches $P_1\cap P_2$ only at $x$. Hence $l$ is the tangential vector of $S_{r_y}(y+r_y\mfk_0)\cap P_1$ at $x$. Since $x$ is the first touching point of $B_{r_y}(y+r_y\mfk_0)$ with $\Sigma$ from the interior, we see that $l$ is also the tangential vector of $\partial \Sigma^1=\S\cap P_1$ at $x$ (see \cref{Fig1-thm1.5}).
  \begin{figure}[H]
	\centering
	\includegraphics[width=8cm]{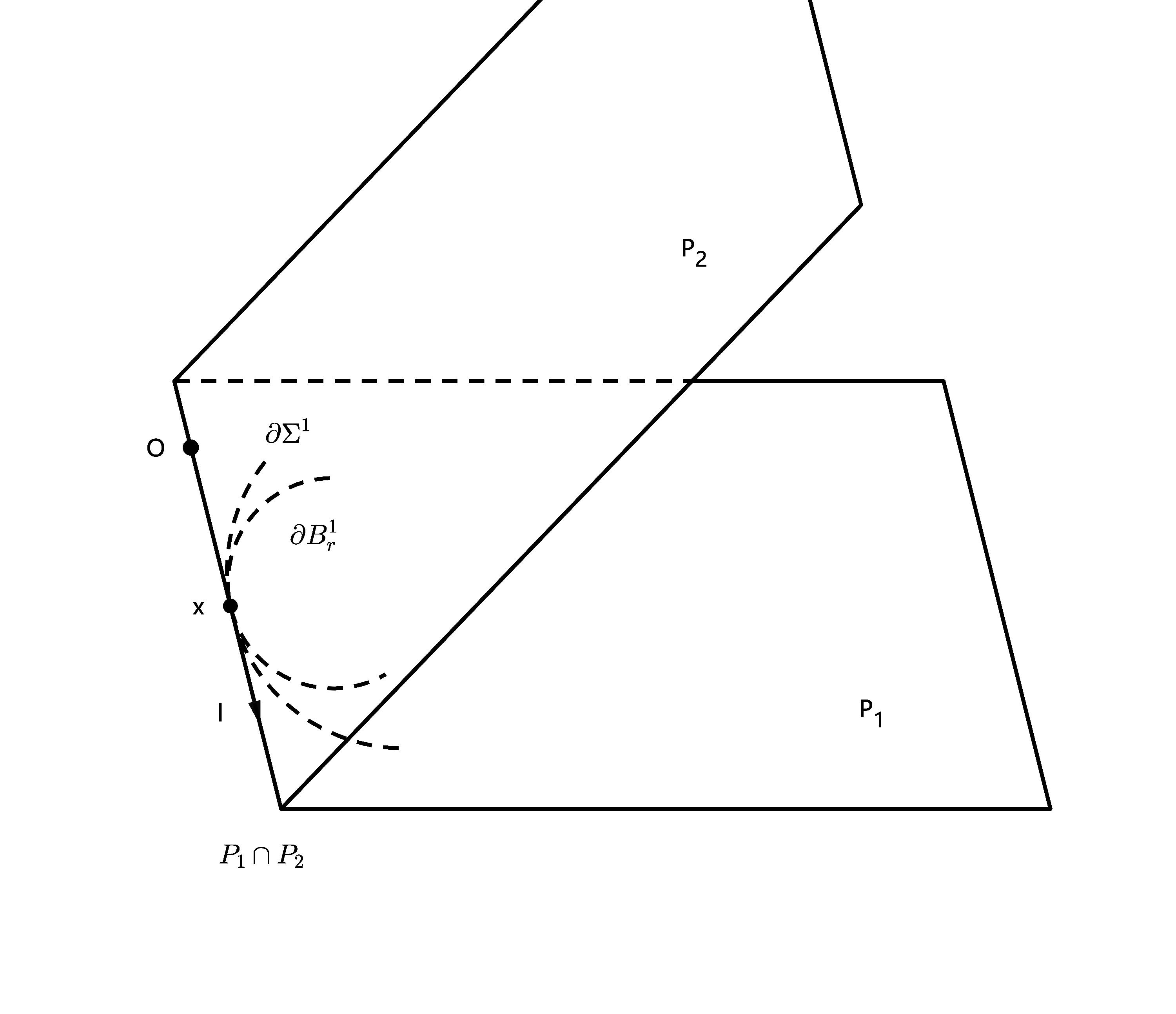}
	\caption{}
	\label{Fig1-thm1.5}
\end{figure}
It follows that $S_{r_y}(y+r_y\mfk_0)\cap P_1$ is tangent to $\partial\S^1$ at $x$. 
 We are now in the same situation as in {\bf Case 2}, \cref{Thm-HK-halfspace} or \cref{Thm-HK-wedge}. Hence the contact angle of $S_{r_y}(y+r_y\mfk_0)$ with $P_1$ at $x$ must be smaller than $\theta^1(x)$, which is a contradiction, since the contact angle of $S_{r_y}(y+r_y\mfk_0)$ with $P_1$ is $\theta^1_0 \ge \theta^1(x)$ by assumption.
 %Following the same argument in {\bf Case 2}, \cref{Thm-HK-halfspace} or \cref{Thm-HK-wedge}, we get a contradiction.

\
 
{\bf Case 2.} Neither $\bar N_i$ is parallel to $\nu_{B_r}(x)$, i.e., $\nu_{B_r}(x)\wedge \bar N_i\neq 0$ for $i=1,2$.

In this case,  $S_{r_y}(y+r_y\mfk_0)\cap P_i(i=1,2)$ must be circles.

\

{\bf Case 2.1.} $P_1\cap P_2\subset T_x \Sigma $.

Since $P_1$ is not parallel to $P_2$, 
there exists exactly one of $P_i (i=1,2)$, say $P_1$, which does not coincide with $T_x\Sigma$. Then $\Sigma$ is  not tangent to $P_1$ at $x$ and $\partial \S^1+\Sigma\cap P_1$ is a curve near $x$. Since $l\in T_x\Sigma$ and $l\perp \bar N_1$, we see that $l$ is the tangential vector of  $\partial\Sigma^1$. 
Since $x$ is the first touching point of $S_{r_y}(y+r_y\mfk_0)$ with $\Sigma$, $l$ is also the tangential vector of $S_{r_y}(y+r_y\mfk_0)\cap P_1$. { Hence $S_{r_y}(y+r_y\mfk_0)\cap P_1$ is tangent to $\S\cap P_1$ at $x$.}
We are again in the position as in {\bf Case 2}, \cref{Thm-HK-halfspace} or \cref{Thm-HK-wedge}, and consequently we get a contradiction.

%When the intersection of $B_r(y + r\mfk_0)$ and
%$P_1$ near $x$ contains only the point $x$, we have $\nu_{B_r}$ is parallel to $N_1$. Thus the intersection of $B_r(y+r\mfk_0)$ and $P_1$ near $x$ must be a curve, and  $l$ is the tangential vector of the intersection of $\Sigma$ and $P_2$ at $x$,

\

{\bf Case 2.2.} $P_1\cap P_2\not\subset T_x\Sigma$.

In this case, $\bar N_i$, $i=1, 2$, cannot be parallel to $\nu(x)$. 
For simplicity of notation, in the following we use $\nu, \nu_{B_r}, \theta^i$ to indicate $\nu(x), \nu_{B_r}(x), \theta^i(x)$, respectively, and 
we adopt the notation $$\partial \Sigma^i=\Gamma_i=\Sigma\cap P_i,\quad\partial B_r^i= S_{r_y}(y+r_y\mfk_0)\cap P_i. $$
Let $T_{\partial\Sigma^i}$ be the unit tangent vector of $\partial \Sigma^i$ at $x$ such that $\<T_{\partial\Sigma^i},\bar{N_j}\> >0$ for $i\not=j$ and $T_{\partial B_r^i}$ be the unit tangent vector of  $\partial B_r^i$ at $x$ such that $\<T_{\partial B_r^i},\bar{N_j}\> >0$ for $i\not=j$. 
Since $\nu$ and $\bar{N_i}$ are perpendicular to $T_{\partial\Sigma^i}$, we know that
$T_{\partial\Sigma^i}$ is parallel to $\nu\wedge \bar{N_i}$. Similarly, since $\nu_{B_r}$ and $\bar{N_i}$ are perpendicular to $T_{\partial B_r^i}$, we have that $T_{\partial B_r^i}$ is parallel to $\nu_{B_r}\wedge \bar{N_i}$.

%{\bf Case 3.2.1} There exists a normal of $P_i$ which is parallel to $\nu_{B_r}$.

%In the following, we will consider the situations that
%there exists a normal of $P_i$ which is parallel to $\nu_{B_r}$, which can be excluded by the property of balls.

%Assume $\bar{N_1}$
 %is parallel to $\nu_{B_r}$, then the ball $B_r(y+r k_0)$ touches the plane $P_1$  only at the point $x$. Therefore the intersection of $B_r(y+r k_0)$ and $P_2$ touches $P_1\cap P_2$  at only the point $x$, and $l$ is the tangential vector of the intersection of $B_r(y+r k_0)$ and $P_2$. Since neither normal of $P_i (i=1,2)$ is parallel to $\nu$, $l$ is not the tangential vector of the intersection of $\Sigma$ and $P_2$, which is contradiction to the assumption that $x$ is the first touching point of $B_r(y+ r k_0)$ with $\bar\Sigma$ from the interior.

%{\bf Case 3.2.2.} Neither normal of $P_i$ is parallel to $\nu_{B_r}$.

Without loss of generality, we may assume that the origin $0\in  \partial \O$ and $x\ne 0$. Let $l$ be the unit tangent vector $\frac{x}{|x|}$.
 With such choice of $l$, we claim that $\<l, \nu\>> 0$ and $\<l,\nu_{B_r}\>>0$. Indeed, recall that by capillarity, $\nu$ is expressed as
  \begin{align*}
      \nu=-\cos \theta^i \bar N_i+\sin \theta^i \bar \nu_i,\quad i=1,2. 
  \end{align*}
  Here  $\bar \nu_i$ is the outward unit normal vector of $\partial \Sigma^i$ in $P_i$ at $x$.
  Since $\partial \Sigma^i$ lies on the same side of the line $P_1\cap P_2$ in $P_i$, it yields that $\<l,\bar\nu_i\> >0$, thus we get  
  \begin{align*}
  \<l, \nu\>=-\cos\theta^i \<l,\bar{N_i}\>+\sin\theta^i\<l,\bar{\nu_i}\>=\sin\theta^i\<l,\bar{\nu_i}\>> 0.
\end{align*}
We will show that $\<l,\nu_{B_r}\> >0$.
Since  $S_{r_y}(y+r_y\mfk_0)$ touches $\S$ from the interior at $x$, we have
$B_{r_y}(y+r_y\mfk_0)\cap P_i\subset \bar\Omega\cap P_i$ for $i=1,2$.
Combining with $0\in\bar \Omega\cap P_1\cap P_2$, it follows that $l$ is pointing outward of $B_r$ at $x$, see \cref{Fig2-thm1.5}, thus
 we obtain $\<l,\nu_{B_r}\>>0$.
   \begin{figure}[H]
	\centering
	\includegraphics[width=12cm]{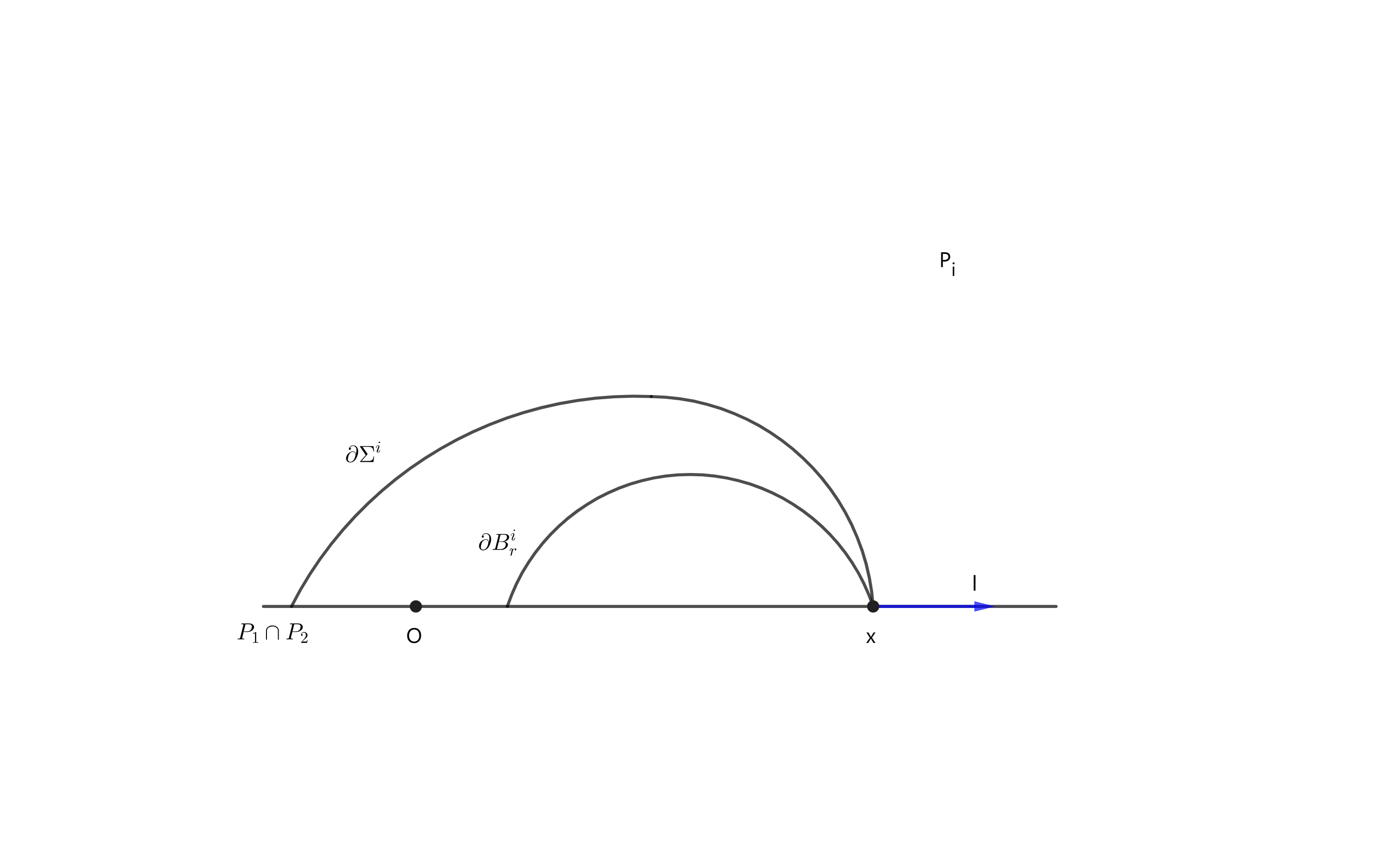}
	\caption{}
	\label{Fig2-thm1.5}
\end{figure}

Note that $\bar{N_1}$ and $\bar{N_2}$ are perpendicular to $l$. Without loss of generality, we may assume 
$$l=-\frac{\bar{N_1}\wedge\bar{ N_2}}{\vert \bar{N_1}\wedge \bar{N_2}\vert}.$$
A direct computation then yields
\begin{align*}
    \<\nu\wedge\bar{N_1},\bar{N_2}\>=\<\bar{N_1}\wedge\bar{N_2},\nu\>=-\vert\bar{N_1}\wedge\bar{N_2}\vert\<l,\nu\> <0,
\end{align*}
 which implies the following fact: the two vectors $T_{\partial\Sigma^1}$ and $\nu\wedge \bar{N_1}$ are in the opposite direction. Therefore,
\begin{align*}
T_{\partial \Sigma^1}=-\frac{\nu\wedge\bar{N_1}}{\vert \nu\wedge\bar{N_1}\vert}.
\end{align*}
By a similar argument, we obtain
\begin{align*}
T_{\partial \Sigma^2}=\frac{\nu\wedge\bar{N_2}}{\vert \nu \wedge\bar{N_2}\vert},\quad
T_{\partial B_r^1}=-\frac{\nu_{B_r} \wedge\bar{N_1}}{\vert \nu_{B_r} \wedge\bar{N_1}\vert},\quad 
T_{\partial B_r^2}=\frac{\nu_{B_r}\wedge\bar{N_2}}{\vert \nu_{B_r}\wedge\bar{N_2}\vert}.
\end{align*}
Thanks to the fact that $x$ is  the first touching point, we must have (see \cref{Fig-3.2})
\begin{align}\label{first-touch-ineq}
    \<T_{\partial B_r^i},l\>\geq \<T_{\partial \Sigma^i},l\>, \quad i=1,2. 
\end{align}
  \begin{figure}[H]
	\centering
	\includegraphics[width=12cm]{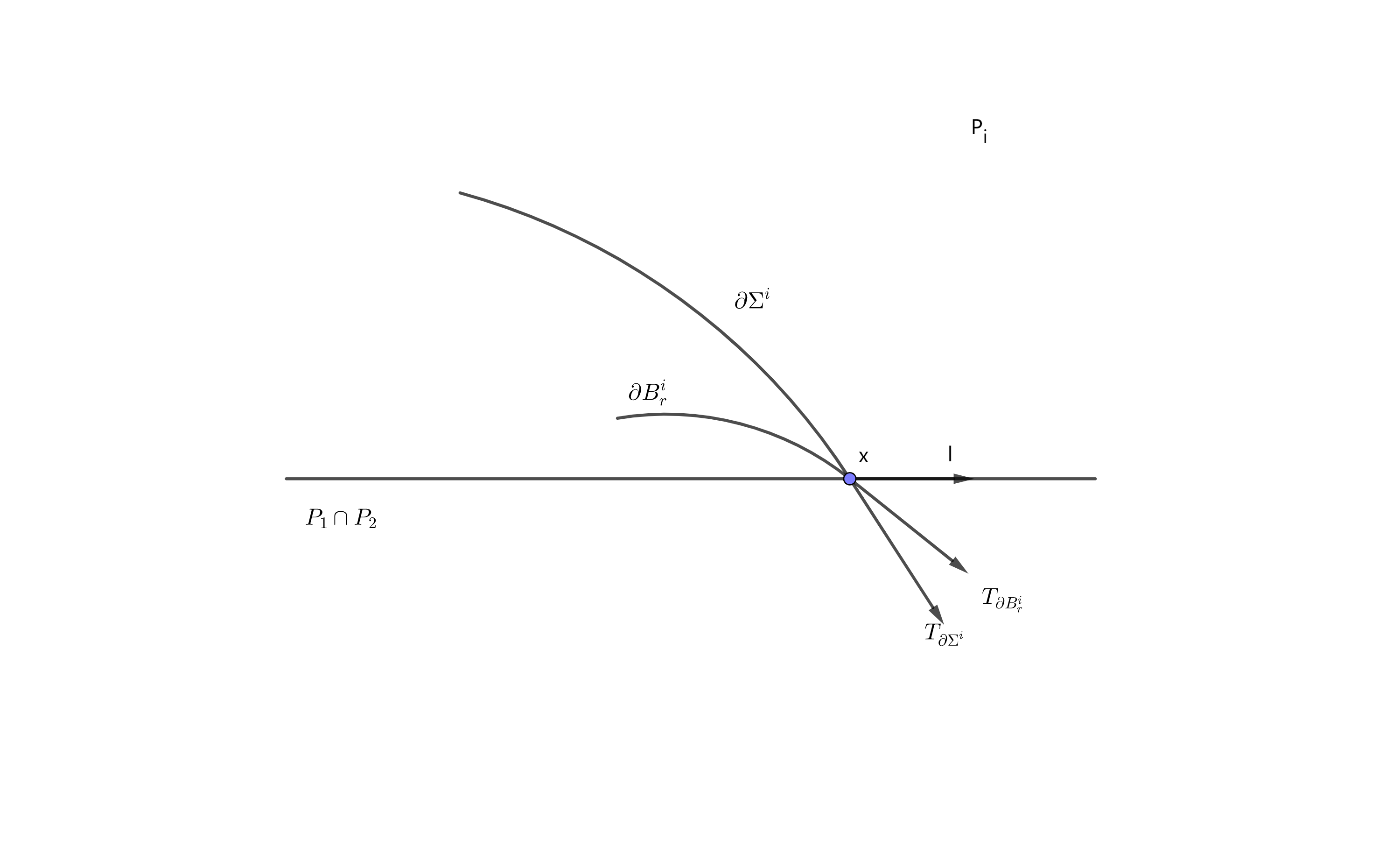}
	\caption{}
	\label{Fig-3.2}
\end{figure}

Let $\eta^i\in (0,\pi)$ be such that $\<\nu_{B_r},\bar{N_i}\>=-\cos\eta^i$ for each $i=1,2$.
We can carry out the following computations.
\begin{align*}
   \<T_{\partial B_r^1},l\> =&\<-\frac{\nu_{B_r}\wedge\bar{N_1}}{\vert\nu_{B_r}\wedge\bar{N_1}\vert},-\frac{\bar{N_1}\wedge\bar{N_2}}{\vert\bar{N_1}\wedge\bar{N_2}\vert}\>\notag\\ 
   =&-\frac{\<\nu_{B_r},\bar{N_2}\>-\<\nu_{B_r},\bar{N_1}\>\<\bar{N_1},\bar{N_2}\>}{\vert \nu_{B_r}\wedge \bar{N_1}\vert\cdot
   \vert \bar{N_1}\wedge\bar{N_2}\vert}\notag\\
   =&\frac{\cos\eta^2+\cos\eta^1\cos\alpha}{\sin\eta^1\sin\alpha}.
\end{align*}
Similarly,
\begin{align*}
   \<T_{\partial B_r^2},l\>  
   =\frac{\cos\eta^1+\cos\eta^2\cos\alpha}{\sin\eta^2\sin\alpha},
\end{align*}
\begin{align*}
\<T_{\partial \Sigma^1},l\> =\frac{\cos\theta^2+\cos\theta^1\cos\alpha}{\sin\theta^1\sin\alpha},
\end{align*}
\begin{align*}
\<T_{\partial\Sigma^2},l\>=\frac{\cos\theta^1+\cos\theta^2\cos\alpha}{\sin\theta^2\sin\alpha}.
\end{align*}
Plugging into \eqref{first-touch-ineq}, we thus obtain
\begin{align}\label{ineq1-first-touch}
    \frac{\cos\eta^2+\cos\eta^1\cos\alpha}{\sin\eta^1\sin\alpha}\geq\frac{\cos\theta^2+\cos\theta^1\cos\alpha}{\sin\theta^1\sin\alpha},
\end{align}
\begin{align}\label{ineq2-first-touch}
    \frac{\cos\eta^1+\cos\eta^2\cos\alpha}{\sin\eta^2\sin\alpha}\geq\frac{\cos\theta^1+\cos\theta^2\cos\alpha}{\sin\theta^2\sin\alpha}.
\end{align}
A crucial observation is that \eqref{ineq1-first-touch} is equivalent to (see the end of \cref{App-1})
\begin{align}\label{ineq1.1-first-touch}
    \sin\theta^1 \big(\cos\eta^2&-\cos(\theta^2-\theta^1+\eta^1)\big)\notag\\
    +&\big(\cos(\theta^2-\theta^1)+\cos\alpha\big)\sin(\theta^1-\eta^1)\geq 0.
\end{align}
%To  estimate the sign of the second term in \eqref{ineq1.1-first-touch},

On the one hand, we have
\begin{align}\label{ineq-position}
    \<x,\bar{N_1}\>
    =&\<x-(y+r_y \mfk_0),\bar{N_1}\>+\<y+r_y \mfk_0,\bar{N_1}\>\notag\\
    =&r_y \<\nu_{B_r},\bar{N_1}\>+\<y,\bar{N_1}\>+r_y \<\mfk_0,\bar{N_1}\>.
\end{align}    
  Since $x\in P_1$ and $y\in\mathring{\mfW}$,
  we have: $\<x,\bar{N_1}\>=0$ and $\<y,\bar{N_1}\><0$. It follows from \eqref{ineq-position}  that
  \begin{align*}
   \<\nu_{B_r},\bar{N_1}\>>-\<\mfk_0,\bar{N_1}\>=-\cos \theta_0^1.
  \end{align*}
By our angle assumption $\theta^1\leq\theta_0^1$, we get
\begin{align*}
      -\cos\eta^1=\<\nu_{B_r},\bar{N_1}\> >-\cos\theta_0^1\geq-\cos\theta^1,
\end{align*}
which implies \begin{align}\label{xeq2}
\eta^1>\theta^1.
\end{align}
On the other hand, since $P_1\cap P_2\not\subset T_x\Sigma$, by \cref{Lem-k<=1}, we see that $ |\mfk(x)| <1$, where 
$\mfk(x):=\sum\limits_{i=1}^2 c_i(x)\bar N_i$,
satisfying $\<\mfk(x), \bar N_i\>=\cos\theta^i$. In view of \cref{lem-condition}, it implies that
\begin{align}\label{xeq3}
    \cos(\theta^2-\theta^1)+\cos\alpha>0.
\end{align} 
From \eqref{xeq2} and \eqref{xeq3}, we know that $$\big(\cos(\theta^2-\theta^1)+\cos\alpha\big)\sin(\theta^1-\eta^1)<0.$$ Combining with  \eqref{ineq1.1-first-touch}, we find 
\begin{align}\label{ineq1-contradition}
    \cos\eta^2-\cos(\theta^2-\theta^1+\eta^1)>0.
\end{align}
By a similar argument we obtain from \eqref{ineq2-first-touch} that
\begin{align}\label{ineq2-contradition}
\eta^2>\theta^2, \end{align} 
and
\begin{align}\label{ineq3-contradition}
    \cos\eta^1-\cos(\theta^1-\theta^2+\eta^2)>0.
\end{align} 
We may assume $\theta^2\geq\theta^1$ without loss of generality, then it follows from \eqref{ineq2-contradition} that 
\begin{align*}
  0<\theta^1<\theta^1-\theta^2+\eta^2\leq \eta^2<\pi.
\end{align*} 
Therefore, back to \eqref{ineq1-contradition}, we deduce that
\begin{align}\label{xeq4}
    \eta^1<\theta^1-\theta^2+\eta^2,
\end{align}
In the meanwhile, from \eqref{ineq3-contradition}, we deduce
\begin{align}\label{xeq5}
\eta^2<\theta^2-\theta^1+\eta^1.
\end{align}
Apparently, \eqref{xeq4} and \eqref{xeq5} lead to a contradiction.

In conclusion, we have showed that the first touching point cannot occur at any $x\in \S\cap P_1\cap P_2$. The rest of the proof follows similarly from \cref{Thm-HK-wedge}.
\end{proof}
\begin{remark}\label{higher-dim}
\normalfont
We make a remark for the case of higher dimensions $\mathbb{R}^{n+1}$. In this case, $P_1\cap P_2$ is a $(n-1)$-plane and $\p\S\cap P_1\cap P_2$ is a closed $(n-2)$-submanifold, playing the role as $\p(\p\S\cap P_1)$ or $\p(\p\S\cap P_2)$. By letting $l$ be the unique unit vector in $P_1\cap P_2$ which is orthogonal to $\p\S\cap P_1\cap P_2$, and $T_{\p\S^i}$ be the tangent vector of $\p\S\cap P_i$ at $x$ which is orthogonal to $\p\S\cap P_1\cap P_2$, and $T_{\p B_r^i}$ be the tangent vector of $\p B_r\cap P_i$ at $x$ which is orthogonal to $\p\S\cap P_1\cap P_2=\p B_r\cap P_1\cap P_2$ ( because $x$ is the first touching point), we may reduce the problem to the $3$-dimensional case. In other words, all the vectors at $x$ we considered in the proof lie in the orthogonal $3$-subspace of $T_x(\p\S\cap P_1\cap P_2)$.
\end{remark}

%{\color{black} Conversely, when $\S$ is a $\theta_\alpha$-capillary spherical cap,  the mean curvature $H$ is a constant and hence from Minkowski's formula, we see equality in \eqref{EQ-HK-wedge} holds. }

%-----------------------------------
%\section{Existence of elliptic point for capillary hypersurface}

\

%-----------------------------------
\section{Alexandrov-Type Theorem}\label{Sec5}
It is well-known that there exists at least one elliptic point for a closed embedded hypersurface in $\mathbb{R}^{n+1}$. Recall that an elliptic point is a point at which the principal curvatures  are all positive with respect to the outward unit normal. We shall show that this fact  is true for capillary hypersurfaces in the half-space, or in a wedge whenever condition \eqref{wedge-condition} holds. Recall that for a wedge case, the existence of an elliptic point is no longer true if without any restriction. An example is easy to be found among ring type capillary surfaces a wedge. See a figure  in \cite{McCuan97}.  

We first need the following lemma (compare to \cref{Rem-foliation}).
%We first generalize this result to the capillary hypersurfaces in the half-space or a wedge.

\begin{lemma}\label{foliation}
Let $\mfW$ be a generalized wedge and $y\in \bigcap\limits_{i=1}^L P_i$. If $|\mfk_0|<1$, then for any $r>0$, the sphere $S_r(y+r \mfk_0)$ intersects $P_i$ at angle $\theta_0^i$. In particular, in $\overline{\mbR^{n+1}_+}$, the sphere $S_r(y-r\cos\theta_0 E_{n+1})$ intersects $\p \mbR^{n+1}_+$ at angle $\theta_0$.
    \end{lemma}
    \begin{proof} First, if $|\mfk_0|<1$, it is easy to see that for any $r>0$, $S_r(y+r \mfk_0)$ intersects $\mathring{P}_i$.
    The outward unit normal to $S_r(y+r\mfk_0)$ at $x\in \mathring{P}_i$ is given by
    \begin{align*}
    \nu_{B_r}(x)=\frac{x-(y+r\mfk_0)}{r}.
\end{align*}
Since $x-y\in P_i$, we see that
$$\<\nu_{B_r}(x), \bar N_i\>=-\<\mfk_0,\bar N_i\>=-\cos\theta_0^i.$$
\end{proof}

\begin{proposition}\label{Prop-ell
		ipticpoint1}
	Let $\S\subset\overline{\mbR^{n+1}_+}$ be a smooth, compact, embedded $\theta_0$-capillary hypersurface. Then there exists at least one elliptic point on $\S$.
\end{proposition}
\begin{proof}
Let $\O$ be the enclosed region of $\S$ and $\p \mbR^{n+1}_+$, $T=\p\O\cap\p \mbR^{n+1}_+$, we fix a point $y\in \mathring{T}$. Consider the family of the open balls $B_r(y-r\cos\theta_0 E_{n+1})$. By \cref{foliation}, $S_r(y-r\cos\theta_0 E_{n+1})=\p B_r(y-r\cos\theta_0 E_{n+1})$ intersects with $\p \mbR^{n+1}_+$ at the angle $\theta_0$.  
Since $\S$ is compact, for $r$ large enough, $ \S\subset B_r(y-r\cos\theta_0 E_{n+1})$. Hence we can find the smallest $r$, say $r_0>0$, such that  $S_{r_0}(y-r_0\cos\theta_0 E_{n+1})$ touches $\S$ at a first time at some $x\in \S$. For simplicity, we abbreviate $S_{r_0}(y-r_0\cos\theta_0 E_{n+1})$ by $S_{r_0}$.

If $x\in\mathring{\S}$, then $\S$ and $S_{r_0}$ are tangent at $x$. 
If $x\in\p\S$, from the fact that both $\S$ and $S_{r_0}$ intersect with $\p \mbR^{n+1}_+$ at the angle $\theta_0$, we conclude again that   $\S$ and $S_{r_0}$ are tangent at $x$.

In both cases, we have that the principal curvatures of $\S$ at $x$ are bigger than or equal to $1/r_0>0$, which implies that $x$ is an elliptic point.
%thanks to \cref{Lem-principal-direction}, we may assume that the principal directions of $\S$ at $x$ are $e_1,\ldots,e_{n-1},e_n=\mu$, where $\mu$ is the unit conormal of $\p\S$ in $\S$, at $x$. We denote by $\kappa_{i;\S}(x)$ the corresponding principal curvatures and it is apparent that $\kappa_{i;\S}(x)\geq1/r_y>0$, for $i=1,\ldots,n-1$.
%As for the conormal direction, by virtue of the minimality of $r_y$, we know that the touching angle of $S_r$ with $P_1$ at $x$ must be $\theta_0$;
%Precisely, for $S_r$, the outwards pointing unit normal at $x$ is given by 
%\begin{align*}
%    \nu_{S_r}(x)=\frac{x-(y-r_y\cos\theta_0E_{n+1})}{r},
%\end{align*}
%and hence by noticing that $x, y\in\p\mbR^{n+1}_+$, we find
%\begin{align}
%    \left<\nu_{S_r}(x),E_{n+1}\right>=\cos\theta_0.
%\end{align}}
%That is, $T_x\S=T_xS_r$, which implies the following fact: the surface $\S$ curves away from the tangent plane $T_xS_r$ at least as fast as the sphere does, namely, $\kappa_{n;\S}(x)\geq1/r_y$ as well. This shows the ellipticity of $x$ and the proof is finished.
\end{proof}

\begin{proposition}\label{Prop-ell
ipticpoint2}
Let $\mfW\subset\mbR^{n+1}$ be a classical wedge. Let $\Sigma\subset\mfW$ be a smooth, compact, embedded $\vec\theta_0$-capillary surface.
If $\S\cap P_1\cap P_2\neq\emptyset$,
then there exists at least one elliptic point on $\S$.
\end{proposition}

\begin{proof}In view of \cref{higher-dim}, we need only consider $3$-dimensional case.
%We point out that the key point of the proof is the choice of the initial center of the shrinking balls, in different cases.

{\color{black}
Since the corners $\Gamma_i$ ($i=1,2$) are smooth co-dimension two submanifolds in $\p \mbR^{n+1}_+$, $\p\S$ is embedded in $\mbR^{n+1}$ (see \cref{Sec2}), when $\S\cap P_1\cap P_2\neq\emptyset$, we must have that $\p\O\cap P_1\cap P_2$ is also a co-dimension two submanifold in $\mbR^{n+1}$.
This means it has non-trivial interior relative to the topology of $P_1\cap P_2$, therefore we are free to}
choose any $y$ in the interior of $\p\O\cap P_1\cap P_2$ as the initial center of the sphere foliation $\{S_r(y+r\mfk_0)\}_{r\geq0}$. Thanks to \cref{Lem-k0<=1-2}, any such foliation must have a first touching point with $\S$ from outside as $r$ decreases from $\infty$.

We consider the sphere $S_{r_0}(y+r_0\mfk_0)$ which touches $\S$ for the first time at some $x\in \S$. Following the argument in \cref{Prop-ell
ipticpoint1}, we see that, if $x\in\mathring{\S}$ or $x\in\p\S\setminus (P_1\cap P_2)$, then $x$ must be an elliptic point of $\S$. Therefore, it remains to consider the case $x\in \partial\Sigma\cap  (P_1\cap P_2)$.

%We use $\nu_{S_r}(x)$ to denote the outwards pointing unit normal of $S_r$ at $x$, defined as in \eqref{defn-nuSr}.
Since $x,y\in P_1\cap P_2$, we find
\begin{align*}
    \left<\nu_{B_r}(x),-\bar N_1\right>=\cos\theta_0^1=\<\nu(x),-\bar{N_1}\>,\\
    \left<\nu_{B_r}(x),-\bar N_2\right>=\cos\theta_0^2=\<\nu(x),-\bar{N_2}\>.
\end{align*}
 This implies that 
 either {\bf (a)} $\nu(x)=\nu_{B_r}(x)$,
 or {\bf(b)} $\nu(x)=\nu_{B_r}(x)+a l$, where $l$ is a unit vector perpendicular to both $\bar{N_1}$ and $\bar{N_2}$, and
 $a$ is a non-zero constant. 
 
We claim that the situation {\bf (b)} does not occur. Otherwise, by the fact that $\nu(x)$ and $\nu_{B_r}(x)$ are unit vectors, we can deduce that $a=-2\<\nu_{B_r(x)},l\>\neq 0$.  Thus we have
 \begin{align}\label{eq-ellipticity1}
 \<\nu(x),l\>=\<\nu_{B_r}(x)+a l,l\>=-\<\nu_{B_r}(x),l\>\neq 0.
 \end{align} 
 %then
 %\begin{align}\label{eq-ellipticity2}
 %    \<\nu(x),l\>\neq 0.
% \end{align}
% {\color{red}For $|\mfk_0|=1$, spheres cannot interesect with $\S$??? 
%When $L$ only contain one point, note that $x=y$, then $\nu_{S_r}=-\mfk_0$. It follows that $|\mfk_0|=1$, by \cref{{Lem-k<=1}}, we have $P_1\cap P_2\subset T_x \Sigma$. By the argument in the proof of \cref{Thm-HK-wedge-2}, this implies that there exsit at least one  intersection of $\Sigma$ and $P_i$ such that $l$ is its tangential vector, which contradicts to $\<\nu(x),l\>\neq 0$. }
%When there exist at least two points in $L$, 
Note that $x\neq y$ and $x-y\in  \overline\Omega\cap B_{r_0}(y+r_0\mfk_0)$, by the fact that $\nu(x)$ and $\nu_{B_r}(x)$ are outward normal vectors, we have $$\<\nu(x) ,x-y\>\geq 0, \quad \<\nu_{B_r}(x),x-y\>\geq 0.$$ Since $x,y\in P_1\cap P_2$,
%the vector $x-y$ is perpendicular to both $\bar{N_1}$ and $\bar{N_2}$, then
$l$ is parallel to $x-y$. It follows that $\<\nu(x),l\>$ and $\<\nu_{B_r}(x), l\>$ have the same sign, which contradicts \eqref{eq-ellipticity1} and concludes the claim.

From the claim, we see that only the situation {\bf (a)} happens.
This means, $S_r$ and $\S$ are tangent at $x$.  %By virtue of \cref{Lem-principal-direction}, $\mu_1(x),\mu_2(x)$ are the principal directions of $\S$ at $x$, and we definitely have ${\kappa}_{i;\S}(x)\geq1/r_y$ for $i=1,2$. 
Thus the principal curvatures of $\S$ at $x$ are bigger than or equal to $1/r_0>0$, which implies $x$ is an elliptic point.
\end{proof}
%-----------------

In view of the proof of \cref{Prop-ell
ipticpoint2}, the only place where we used the condition $\S\cap P_1\cap P_2\neq\emptyset$ is that we have to use it in \cref{Lem-k0<=1-2}, as $\vert\mfk_0\vert=1$, to conclude that the sphere foliation $\{S_r(y+r\mfk_0)\}_{r\ge0}$ must touch $\S$. In this regard, we can remove the extra assumption $\S\cap P_1\cap P_2\neq\emptyset$
%and allowing the dimension of the ambient Euclidean space to be $n+1\geq3$,
by strengthening the angle condition to be $\vert\mfk_0\vert<1$. Indeed, we have the following
\begin{proposition}
\label{Cor-ellipticpoint}
Let $\mfW\subset\mbR^{n+1}$ be a classical wedge. Let $\S\subset\mfW$ be a smooth, compact, embedded $\vec\theta_0$-capillary hypersurface with $\vert\mfk_0\vert<1$.
Then there exists at least one elliptic point on $\S$.
\end{proposition}

If the hypersurfaces are CMC hypersurfaces, we have 
% we relax the constraint \eqref{wedge-condition} to a weaker condition $\pi-\alpha
%\leq \theta_0^1+
%\theta_0^2$.  (Is it true? \textcolor{red}{Does the proof implies the existence of elliptic points?}
%On the other hand, we can use foliation of hyperplanes to get the following result for CMC case.
\begin{proposition}
\label{Lem-H>0}
Let $\mfW\subset\mbR^{n+1}$ be a classical wedge. Let $\S\subset\mfW$ be a smooth, compact, embedded $\vec\theta_0$-capillary hypersurface.
If $\S$ is of constant mean curvature $H$, then $H>0$, provided $\pi-\alpha\leq\theta_0^1+\theta_0^2.$
\end{proposition}
\begin{proof}We remark that the free boundary case, that is $\vec\theta_0=(\frac{\pi}{2},\frac{\pi}{2})$ has been proved by Lopez \cite{Lopez14}. We consider here the general case.

To begin, we take $p_0\in\p \S$ to be the point of maximal distance of $\p \S$ from the edge $P_1\cap P_2$. Assume without loss of generality that $p_0\in P_2$. Let $\mathcal P$ be the family of planes  parallel  to the edge of $\mfW$ and  having contact angle $\theta_0^2$  with $P_2$. Starting from one of such a plane near infinite and  moving it among this family % towards $\O$ parallely  parallel displace $P$ sufficiently far from the edge so that $P\cap\S=\emptyset$, 
%then we move it back parallelly 
until the first time that one of plane $P$ in $\mathcal P$ touches $\S$. By definition of $p_0$ and $P$, it is only possible that the first touching occurs at certain interior point of $\S$, at $p_0\in P_2$ or at some $p_1\in \S\cap P_1$.

For the former two cases, one can use the strong maximum principle for the interior point and the Hopf-lemma for the boundary point, and obtain $H>0$. For the last case, we shall use our angle assumption $\pi-\alpha\leq\theta_1+\theta_2$.

Indeed, if $\pi-\alpha<\theta_1+\theta_2$, then the first touching point must not occur at any points of $\S\cap P_1$: since $P$ is a plane having contact angle $\theta_0^2$ with $P_2$, we know that the contact angle of $\S$ with $P_1$, say $\theta_P^1$, is $\pi-\alpha-\theta_0^2$. By the angle assumption, we find
\begin{align}
    \theta_P^1=\pi-\alpha-\theta_0^2<\theta_0^1,
\end{align}
which is not possible if $P$ touches $\S\cap P_1$ from outside for the first time.

If $\pi-\alpha=\theta_0^1+\theta_0^2$, then $\theta_P^1=\theta_0^1$ from the  above discussion, and we can use the Hopf's boundary point lemma again to find that $H>0$. This completes the proof.
\end{proof}

Now, we are in the position to prove the Alexandrov-type theorem and the non-existence theorem.
%-----------------
\begin{proof}[Proof of \cref{Thm-Alexandrov}]
    Before we proceed the proof, we emphasize that the condition $\vert\mfk_0\vert\leq1$ ensures the validity of the inequality $$1+\left<\nu,\mfk_0\right>\geq0$$
    pointwisely on $\S$. In particular, as $L=1$, we have
    \begin{align*}
        1-\cos\theta_0\left<\nu,E_{n+1}\right>\geq0
    \end{align*}
    along $\S$.
    
On  one hand, by virtue of \cref{Prop-ell
ipticpoint2} and G\"arding's argument \cite{Garding59} (see also \cite[Section 3]{Ros87}), we know that $H_j$ are positive, for $j\le r$. %The Maclaurin inequality then yields
%\begin{align}\label{ineq-Hr<=H1}
  %  H_r^{1/r}\leq H_1=\frac{H}{n}.
%\end{align}
    Applying \cref{Thm-HK-wedge-2} and using the Maclaurin inequality $$H_1\ge H_r^{1/r},$$ %\eqref{ineq-Hr<=H1},
    we find
    \begin{align}\label{eq-Alex-1} (n+1)H_r^{1/r}\vert\O\vert\le H_r^{1/r}\int_\S\frac{1+\left<\nu,\mfk_0\right>}{H/n}\rd A\leq\int_\S\left(1+\left<\nu,\mfk_0\right>\right)\rd A,
    \end{align}
    and equality holds if and only if $\S$ is a ${\theta}_0$-capillary spherical cap.
    
    On the other hand, %using the Maclaurin inequality again, we have
   % $$H_{r-1}\geq H_r^{\frac{r-1}{r}}.$$
    using the Minkowski formula \eqref{Minkowski-wedge} and the Maclaurin inequality again, we have
    \begin{align*}%\label{eq-Alex-2}
        0
        =&\int_\S\left(1+\left<\nu,\mfk_0\right>\right)H_{r-1}-H_r\left<x,\nu\right>\rd A\notag\\
        \geq&\int_\S\left(1+\left<\nu,\mfk_0\right>\right)H_r^{\frac{r-1}{r}}-H_r\left<x,\nu\right>\rd A\notag\\
        =&H_r^{\frac{r-1}{r}}\int_\S1+\left<\nu,\mfk_0\right>-H_r^{1/r}\left<x,\nu\right>\rd A
        \\
        =&H_r^{\frac{r-1}{r}}\left\{\int_\S1+\left<\nu,\mfk_0\right>\rd A- (n+1)H_r^{1/r} |\Omega|\right\},
        \end{align*}
    where in the last equality we have used \eqref{mink0}.
   % Using the divergence theorem for the position vector field $X(x)=x$ over $\O$, we o
    %  The proof of \eqref{eq-Alex-3} follows from the divergence theorem applied on $\div x=n+1$.
   % where we have used the fact that $\left<x,\bar N_i\right>=0$ on $T_i$.
%   \begin{align*}    &\int_\S\left(1+\left<\nu,\mfk_0\right>\right)dA-(n+1)H_r^{1/r}\vert\O\vert\notag\\
       % =&\int_\S1+\left<\nu,\mfk_0\right>-H_r^{1/r}\left<x,\nu\right>dA\leq0.
 %   \end{align*}
    Thus equality in \eqref{eq-Alex-1} holds, and hence $\S$ is a $\vec{\theta}_0$-capillary spherical cap. This completes the proof.
\end{proof}

\begin{proof}[Proof of \cref{Thm-non-exist}]
For the CMC case, we notice that our condition \eqref{wedge-condition} implies $\pi-\alpha\leq\theta_0^1+\theta_0^2$ automatically, thanks to \cref{lem-condition}. Therefore, we can use \cref{Lem-H>0} to see that $\S$ is of positive constant mean curvature. 

For the constant $H_r$ case, since we assume $\vert\mfk_0\vert<1$, we may use \cref{Cor-ellipticpoint} to conclude that $\S$ has an elliptic point, from which we see that $H_r$ is a positive constant.

In view of this, arguing as in the proof of \cref{Thm-Alexandrov}, we can use \cref{Prop-Minkowski} together with \cref{Thm-HK-wedge}  to show that the equality case happens in the Heintze-Karcher inequality \eqref{EQ-HK-wedge}, and hence $\S$ must be a $\vec\theta_0$-spherical cap. However, if this is the case, a simple geometric relation (see {\cref{Fig-thm1.7}}) then implies that $\alpha+(\pi-\theta_0^1)+(\pi-\theta_0^2)<\pi$, i.e., $\pi+\alpha<\theta_0^1+\theta_0^2$, which is incompatible with \eqref{wedge-equiv}.
The proof is complete.
\end{proof}
%in other words, none smooth, compact, embedded, $\vec\theta_0$-capillary CMC hypersurface shall exist, under the conditions: {\bf 1.} $\vert\mfk_0\vert\leq1$; {\bf 2.} $\S\cap P_1\cap P_2=\emptyset$. 
\begin{figure}[H]
	\centering
	\includegraphics[width=12cm]{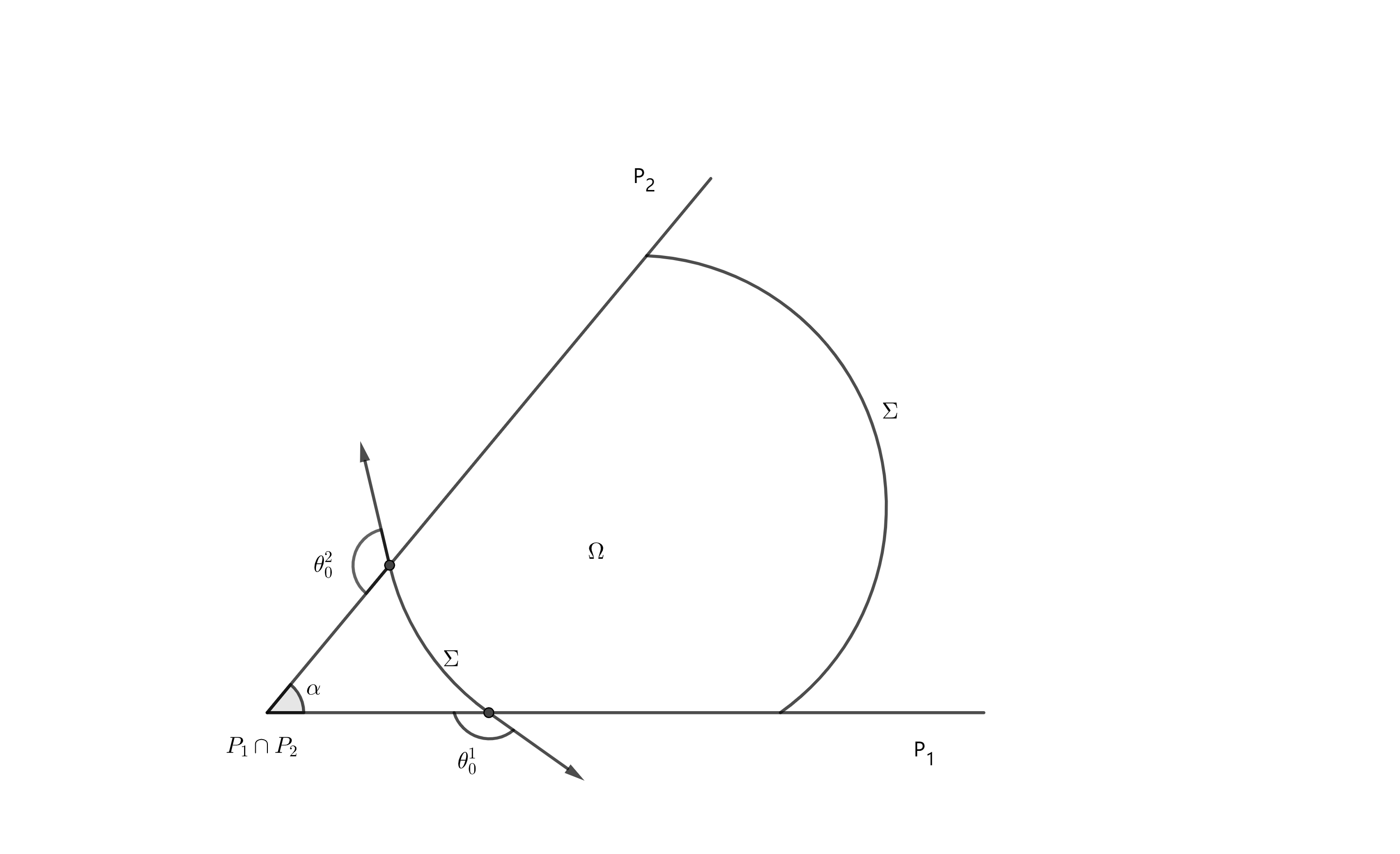}
	\caption{}
	\label{Fig-thm1.7}
\end{figure}

%%%%%%%%%%%%%%%%%%%%%%%%%%%%%%%%%%%%%%%%%%%%%%%%%%%%%%%%%
%https://www.overleaf.com/project/6239e84177180e555c09c8cd
\

\appendix
\section{Miscellaneous Results in Wedge}\label{App-1}

\begin{lemma}\label{lem-condition}
For the case $L=2$, namely, $\mfW\subset\mbR^{n+1}$ is a classical wedge,  \eqref{wedge-condition} is equivalent to \eqref{wedge-equiv}, i.e.,
\begin{align*} %\label{wedge-equiv}
    \vert\pi-(\theta_0^1+\theta_0^2)\vert\le\alpha\le\pi-\vert\theta_0^1-\theta_0^2\vert.
   %|\theta_1-\theta_2|\leq\pi-\alpha\leq \theta_1+\theta_2\leq \pi+\alpha
\end{align*}
Similarly, \eqref{wedge-condition} with  strict inequality is equivalent to \eqref{wedge-equiv} with  strict inequality.
\end{lemma}

\begin{proof}
First we note that \eqref{wedge-equiv} can be viewed of somewhat a compatible condition for admitting spherical caps in a wedge(a wedge is determined by its dihedral angle $\alpha$.) with prescribed angles $\vec{\theta_0}=(\theta^1_0,\theta^2_0)$.

By definition of $\mfk_0$ and the fact that $\left<\bar N_1,\bar N_2\right>=\cos({\pi-\alpha})=-\cos\alpha$, we find
\begin{align*}
    \begin{cases}
    c_1-c_2\cos\alpha&=\cos\theta^1_0,\\
    -c_1\cos\alpha+c_2&=\cos\theta^2_0.
    \end{cases}
\end{align*}
A simple computation then yields 
\begin{align*}%\label{eq-c1c2}
    \begin{pmatrix}
    c_1\\
    c_2
    \end{pmatrix}
    =\frac{1}{\sin^2\alpha}
    \begin{pmatrix}
    1&\cos\alpha\\
    \cos\alpha&1
    \end{pmatrix}
    \begin{pmatrix}
    \cos\theta^1_0\\
    \cos\theta^2_0
    \end{pmatrix},
\end{align*}
and it follows that
\begin{align*}
    \vert \mfk_0\vert^2
    =&\left<\mfk_0,c_1\bar N_1+c_2\bar N_2\right>
    =c_1\cos\theta^1_0+c_2\cos\theta^2_0\notag\\
    =&\begin{pmatrix}
    \cos\theta^1_0&\cos\theta^2_0
    \end{pmatrix}
    \begin{pmatrix}
    c_1\\
    c_2
    \end{pmatrix}\notag\\
    =&\frac{\cos^2{\theta_0^1}+\cos^2\theta_0^2+2\cos{\theta_0^1}\cos{\theta_0^2}\cos{\alpha}}{\sin^{2}{\alpha}}.
\end{align*}
 Thus \eqref{wedge-condition} can be rewritten as 
 \begin{align*}
      \cos^2{\theta_0^1}+\cos^2\theta_0^2+2\cos{\theta_0^1}\cos{\theta_0^2}\cos{\alpha}\leq\sin^{2}{\alpha}.
 \end{align*}
which is 
\begin{align*}
    \big(\cos\alpha+\cos(\theta_0^1+\theta_0^2)\big)\big(\cos\alpha+\cos(\theta_0^1-\theta_0^2)\big)\leq 0.
\end{align*}
Thus we see that \eqref{wedge-condition} is equivalent to 
\begin{align*}
 |\theta_0^1-\theta_0^2|\leq\pi-\alpha\leq \theta_0^1+\theta_0^2\leq \pi+\alpha,
\end{align*}
which is just \eqref{wedge-equiv}.
\end{proof}

The following observation is important for our analysis in \cref{Thm-HK-wedge-2}. 
\begin{lemma}\label{Lem-k<=1}
Let $\mfW\subset\mbR^{n+1}$ be a classical wedge.
If $\S\subset\mfW$ is a smooth, compact, embedded $\vec\theta$-capillary hypersurface.  Let $\mfk:\p\S\to \rr^{n+1}$ be given by $\mfk(x)=\sum\limits_{i=1}^2c_i(x)\bar N_i$ such that $\<\mfk(x),\bar N\>=\cos\theta^i(x).$
If $\Sigma\cap P_1\cap P_2 \neq\emptyset$, then we have $\vert \mfk\vert\leq 1$ on $\Sigma\cap P_1\cap P_2$.
Moreover, for any $x\in\S\cap P_1\cap P_2$,
$P_1\cap P_2\subset T_x\Sigma$ if and only if $|\mfk(x)|=1$.
\end{lemma}
\begin{proof}In view of \cref{higher-dim}, we need only consider the $3$-dimensional case.

Let $x\in\S\cap P_1\cap P_2$, In the following, we compute at $x$. We have
$$\<\bar{N_1},\bar{N_2}\>=-\cos \alpha,\quad \<\nu,\bar{N_i}\>=-\cos\theta^i,\quad i=1,2. $$
Thus
\begin{align*}
 \<\bar{N_1}\wedge \nu,\bar{N_2}\wedge \nu\>
 &=\<\bar{N_1},\bar{N_2}\>-\<\bar{N_1},\nu\>\<\bar{N_2},\nu\>\notag\\
 &=-\cos\alpha-\cos\theta^1 \cos\theta^2.
\end{align*}
Since $|\<\bar{N_1}\wedge \nu, \bar{N_2}\wedge \nu\>|\leq \sin\theta^1\sin\theta^2$, we deduce
\begin{align*}
 \vert\cos\alpha+\cos\theta^1 \cos\theta^2 \vert\leq \sin\theta^1\sin\theta^2,  
\end{align*}
which implies $\vert \mfk\vert \leq 1$.

If $P_1\cap P_2\subset T_x \Sigma $, we have $ |\<\bar{N_1}\wedge \nu,\bar{N_2}\wedge \nu\>|=\sin\theta^1\sin\theta^2$. Then we get
\begin{align*}
   |\cos\alpha+\cos\theta^1 \cos\theta^2| =\sin\theta^1\sin\theta^2,
\end{align*}
which is equivalent to $|\mfk(x)|=1$.
%If $P_1\cap P_2\not\subset T_x\Sigma $, we have $ |\<\bar{N_1}\times \nu,\bar{N_2}\times \nu\>|<\sin\theta^1\sin\theta^2,$  which is equivalent to $|k|<1$.
\end{proof}
We point out that, in the discussion above,
if $\S$ intersects $P_1$ and $P_2$ transversally at $x\in P_1\cap P_2$, then \cref{Lem-k<=1} is included in \cite[Lemma 2.5]{Li21}.

\begin{lemma}\label{Lem-k0<=1-1}
Let $\mfW\subset\mbR^{n+1}$ be a classical wedge.
If $\S\subset\mfW$ is a smooth, compact, embedded $\vec\theta_0$-capillary with $\vert\mfk_0\vert\leq1$, then for any $y\in\Omega$, the family of spheres $\{S_r(y+r\mfk_0)\}_{r\geq0}$ must touch $\S$.
\begin{proof}
The case $\vert \mfk_0\vert<1$ follows trivially, since $\{S_r(y+r\mfk_0)\}_{r\geq0}$ foliates the whole $\mbR^{n+1}$. As for $\vert\mfk_0\vert=1$, we proceed by the following observation.

\noindent {\bf Observation.}
Since $y\in S_r(y+r\mfk_0)$ for any $r\geq0$, with $\nu_{S_r}(y)=-\mfk_0$. Moreover,
\begin{align}
    B_r(y+r\mfk_0)\rightarrow H_y^-:=\{z\in\mbR^{n+1}:\left<z-y,\mfk_0\right>>0\}\quad\text{as }r\rightarrow\infty.
\end{align}
In other words, the family of spheres $S_r(y+r\mfk_0)$ foliates the half-space $H^-_y$.

{\color{black}
We claim that for any $y\in\Omega$, there holds $H_y^-\cap\S\neq\emptyset$.
To see this, we consider the following situations separately.

{\bf Case 1.} $\S\cap P_1\cap P_2=\emptyset$.

By definition of $\S$, we see that $\p\Omega=\S\cup T_1\cup T_2$ with $T_1$, $T_2$ away from the edge $P_1\cap P_2$ (see \cref{Fig-thm1.7} for illustration).
Therefore for any $\mfk_0\in\mathbb{S}^n$, the open half-space determined by $y,\mfk_0$ must intersect $\S$.

{\bf Case 2.} $\S\cap P_1\cap P_2\neq\emptyset$.

In this case,
since $\S\cap P_1\cap P_2\neq\emptyset$, we must have $P_1\cap P_2\subset T_x\S$ for any $x\in\S\cap P_1\cap P_2$, according to \cref{Lem-k<=1} (with $\vec\theta=\vec\theta_0$ chosen therein).
However, this contradicts to our assumption on $\S$, precisely, that $\Gamma_i$ $(i=1,2)$ are smooth co-dimension two submanifolds in $\mbR^{n+1}$.
This proves the claim and hence completes the proof.
}

%Since $y\in\Omega$, it is apparent that $H_y^-\cap\S\neq\emptyset$, which yields the proof.
\end{proof}
\end{lemma}
In the proof of \cref{Prop-ell
ipticpoint2}, we use the family of spheres $S_r(y+r\mfk_0)$ for some $y$ in the interior of $\bar\Omega\cap P_1\cap P_2$, the following observation is needed.

\begin{lemma}\label{Lem-k0<=1-2}
Let $\mfW\subset\mbR^{n+1}$ be a classical wedge.
If $\S\subset\mfW$ is a smooth, compact, embedded $\vec\theta_0$-capillary such that $\S\cap P_1\cap P_2\neq\emptyset$, then for any $y$ in the interior of $\p\O\cap P_1\cap P_2$, the family of spheres $\{S_r(y+r\mfk_0)\}_{r\geq0}$ must touch $\S$.
\end{lemma}
\begin{proof}
In view of \cref{higher-dim}, we need only consider the $3$-dimensional case. 

By virtue of \cref{Lem-k<=1}, we have $\vert\mfk_0\vert\le 1$.
The case $\vert\mfk_0\vert<1$ follows trivially from \cref{Lem-k0<=1-1}. As for $\vert\mfk_0\vert=1$, by {\bf Observation} above and the fact that $\<y,\mfk_0\>=0$ due to $y\in P_1\cap P_2$, it suffices to show that $P_1,P_2\subset H_0^-$.

{\bf Claim. }In \eqref{k0}, we have $c_1,c_2<0$, provided that $\vert\mfk_0\vert=1$ and $\S\cap P_1\cap P_2\neq\emptyset.$

If the claim holds, a direct computation yields: for any $z\in \mathring{P_1}$,
\begin{align*}
    \<z,\mfk_0\>=\<z,c_1\bar N_1+c_2\bar N_2\>=c_2\<z,\bar N_2\>.
\end{align*}
Notice that $\{l,l_1,\bar N_1\}$ forms an orthonormal basis of $\mbR^3$, where $l$ is a fixed unit vector, parallel to $P_1\cap P_2$, $l_1$ is the unit inwards pointing conormal of $\p P_1$ in $P_1$. In this coordinate, since $z\in\mathring{P_1}$, it can be expressed as $z=a_0l+a_1l_1+0\bar N_1$ with $a_1>0$. It follows that $\<z,\mfk_0\>=a_1c_2(-\sin\alpha)>0$. Similarly, we have: for any $z\in\mathring{P_2}$, there holds $\<z,\mfk_0\>>0$, which implies that $P_1,P_2\subset H_0^-$ and proves the Lemma.

We are thus left to prove the {\bf Claim.} Indeed, by \cref{Lem-k<=1}, let $x\in\S\cap P_1\cap P_2$, then we have $P_1\cap P_2\in T_x\S$. In view of this, we obtain: $\nu(x)\in{\rm span}\{\bar N_1,\bar N_2\}$, where $\nu(x)$ is the unit outward normal of $\S$ at $x$. Due to the contact angle condition, we must have
\begin{align}
    \<\nu(x),\bar N_i\>=-\cos\theta_0^i,
\end{align}
comparing with the definition of $\mfk_0$ \eqref{k0}, we thus find: $$\nu(x)=-\mfk_0=-c_1\bar N_1-c_2\bar N_2.$$ Since $\nu(x)$ is the outward unit normal of $\S$ at $x$, for any $y_i\in\p\Omega\cap \mathring{P_i}, i=1,2$, we have
\begin{align*}
    &\<x-y_i,\nu(x)\>>0.
\end{align*}
Indeed, it follows from $\nu(x)=-\cos \theta_0^i \bar N_i+\sin \theta_0^i \bar \nu_i(x)$ and the fact $\bar \nu_i(x)$ is orthogonal to $l$ and is pointing outward $P_i$ at $x$.

Meanwhile, since $y_i\in \mathring{P_i}$, we definitely have
\begin{align}
    &\<-y_i,\bar N_i\>=0,\notag\\
    &\<-y_i,\bar N_j\>>0\quad\text{for }j\neq i.\notag
\end{align}
Recall that $x\in P_1\cap P_2$ and hence $\<x,\bar N_i\>=0$ for each $i$, we thus obtain
\begin{align}
    &\<x-y_i,\bar N_i\>=0,\notag\\
    &\<x-y_i,\bar N_j\>>0,\quad\text{for }j\neq i.\notag
\end{align}
Combining all above and invoking that $\nu(x)=-c_1\bar N_1-c_2\bar N_2$, we thus find: $c_1<0, c_2<0$. This proves the {\bf Claim} and completes the proof.

\end{proof}

\noindent{\bf Proof of   \eqref{ineq1-first-touch}$\Leftrightarrow$\eqref{ineq1.1-first-touch}:}
Using
\begin{align*}
\cos\theta^2=\cos(\theta^2-\theta^1)\cos\theta^1-\sin(\theta^2-\theta^1)\sin\theta^1,
\end{align*}
we see that \eqref{ineq1-first-touch} is equivalent to
\begin{align*}
  &(\cos\eta^2+\cos\eta^1\cos\alpha)\sin\theta^1\notag\\
   \geq&\Big(\cos(\theta^2-\theta^1)\cos\theta^1-\sin(\theta^2-\theta^1)\sin\theta^1+\cos\theta^1\cos\alpha\Big)\sin\eta^1.
\end{align*}
  After rearranging, we get
 \begin{align*}&\sin\theta^1\cos\eta^2+\cos\alpha \sin{(\theta^1-\eta^1)}\notag\\=&
     \sin\theta^1\cos\eta^2+\cos\alpha(\sin\theta^1\cos\eta^1-\sin\eta^1\cos\theta^1)\notag\\
     \geq& \cos(\theta^2-\theta^1)(\cos \theta^1\sin\eta^1-\sin\theta^1\cos\eta^1)\notag\\
     &+\sin\theta^1\Big(\cos{(\theta^2-\theta^1)}\cos\eta^1-\sin(\theta^2-\theta^1)\sin\eta^1\Big)
     \\=&-\cos(\theta^2-\theta^1)\sin(\theta^1-\eta^1)
     +\sin\theta^1\cos(\theta^2-\theta^1+\eta^1).
 \end{align*}
 which is just \eqref{ineq1.1-first-touch}.\qed

\

 %----------------
 %\section{}
 %In this section, we collect some known results, which are needed for our discussion on classical wedge.
%\begin{lemma}[{\cite[Lemma 2.5]{Li21}}]\label{Lem-Li2.5}
%    Let $P_1,P_2$ be two half-planes in $\mathbb{R}^{3}$, enclosing a wedge region $\mfW$ with opening angle $\alpha\in(0,\pi)$. Suppose $\Gamma$ is a plane in $\mathbb{R}^{3}$, such that the dihedral angle between $\Gamma$ and $P_i$, $i=1,2$, is $\theta_i\in(0,\pi)$. Then we have
%    \begin{align}
        %\vert\pi-(\theta_1+\theta_2)\vert<\alpha<\pi-\vert\theta_1-\theta_2\vert.
%    \end{align}
%\end{lemma}

%\begin{definition}[Ring type spanner]\label{Defn-ringtype}
%\normalfont
%A ring-type surface is a compact, connected, orientable surface with two boundary components and Euler-Poincar\'e characteristic zero. Ring-type surfaces are also known as topologically annular.
%\end{definition}

%\begin{theorem}[{\cite[Theorem 1]{McCuan97}}]\label{Thm-McCuan}
%If $\theta_1+\theta_2\leq\pi+\alpha$, then there are no embedded ring-type surfaces of constant mean curvature spanning a wedge of opening angle $\alpha$ and having constant contact angles, $\theta_1$ and $\theta_2$, with the wedge on each component of their boundary.
%\end{theorem}

%\begin{theorem}[{\cite[Theorem 1]{Park05}}]\label{Thm-Park}
%Every ring type spanner in a wedge is spherical.
%\end{theorem}

%========
\bibliographystyle{siam}%aplha, abbrv, siam
\bibliography{JWXZ23.bib}

\end{document}